\documentclass[11pt]{amsart}
\usepackage{amsmath, amscd, amsthm, amssymb}

\usepackage{fullpage}
\numberwithin{equation}{section}

\theoremstyle{plain}
\newtheorem{theorem}{Theorem}[section]
\newtheorem{lemma}[theorem]{Lemma}
\newtheorem{proposition}[theorem]{Proposition}

\newtheorem{corollary}[theorem]{Corollary}

\theoremstyle{definition}

\newtheorem{remark}[theorem]{Remark}

\DeclareMathOperator{\charf}{char}

\renewcommand{\(}{\left(}
\renewcommand{\)}{\right)}
\newcommand{\ol}[1]{\overline{#1}}
\newcommand{\set}[1]{\left\lbrace#1\right\rbrace}

\newcommand{\mm}[4]{\left(\begin{smallmatrix} #1 & #2\\ #3 & #4\end{smallmatrix}\right)}
\newcommand{\mb}[4]{\left(\begin{array}{cc} #1 & #2\\ #3 & #4\end{array}\right)}

\DeclareMathOperator{\Sym}{Sym}

\DeclareMathOperator{\ord}{ord}
\DeclareMathOperator{\val}{val}
\DeclareMathOperator{\tr}{tr}
\DeclareMathOperator{\ind}{Ind}
\DeclareMathOperator{\SO}{SO}
\DeclareMathOperator{\Spin}{Spin}
\DeclareMathOperator{\Sp}{Sp}
\DeclareMathOperator{\GSp}{GSp}
\DeclareMathOperator{\PGSp}{PGSp}
\DeclareMathOperator{\SL}{SL}
\DeclareMathOperator{\GL}{GL}
\DeclareMathOperator{\PGL}{PGL}
\DeclareMathOperator{\diag}{diag}

\def\Q{{\mathbf Q}}
\def\Z{{\mathbf Z}}
\def\A{{\mathbf A}}
\def\R{{\mathbf R}}
\def\C{{\mathbf C}}

\begin{document}
\title[The integral of Kohnen and Skoruppa]{On the Rankin-Selberg integral of Kohnen and Skoruppa}
%1991 Mathematics Subject classification: Primary: 11F46 (Siegel modular groups; Siegel and Hilbert-Siegel modular and automorphic forms), 11F66 (Langlands $L$-functions; one variable Dirichlet series and functional equations); Secondary:  22E55  (Representations of Lie and linear algebraic groups over global fields and adèle rings)

\author{Aaron Pollack}
\address{Department of Mathematics, Stanford University, Stanford, CA 94305, USA}\email{aaronjp@stanford.edu}
\author{Shrenik Shah}
\address{Department of Mathematics, Columbia University, New York, NY 10027, USA}\email{snshah@math.columbia.edu}

\thanks{A.P.\ has been partially supported by NSF grant DMS-1401858.  S.S.\ has been supported in parts through NSF grants DGE-1148900, DMS-1401967, and also through the Department of Defense (DoD) National Defense Science \& Engineering Graduate Fellowship (NDSEG) Program.}
\begin{abstract}
The Rankin-Selberg integral of Kohnen and Skoruppa produces the Spin $L$-function for holomorphic Siegel modular forms of genus two.  In this paper, we reinterpret and extend their integral to apply to arbitrary cuspidal automorphic representations of $\PGSp_4$.  We show that the integral is related to a non-unique model and analyze it using the approach of Piatetski-Shapiro and Rallis.
\end{abstract}

\maketitle
%% \tableofcontents %% Just for papers exceeding 50 pages.

\section{Introduction}
The Rankin-Selberg integral representation of Kohnen-Skoruppa \cite{ks} produces the Spin $L$-function for holomorphic Siegel modular cusp forms on $\GSp_4$. Their integral makes use of a special Siegel modular form $P_D(Z)$, which is in the Saito-Kurokawa or Maass subspace.  Here $Z$ is the variable in the Siegel upper half space of genus two, and $D$ is a negative discriminant.  The proof in \cite{ks} is classical, and involves global calculations with the Fourier coefficients of Siegel modular forms.  In particular, Kohnen and Skoruppa make essential use of the Maass relations, which are identities satisfied by the Fourier coefficients of Siegel modular forms that are Saito-Kurokawa lifts.

The purpose of this paper is to reinterpret and extend the integral representation in \cite{ks}.  We define a class of automorphic functions $P_D^\alpha(g)$ on $\GSp_4$ depending on a nonzero square-free integer $D$ (which is allowed to be positive) and auxiliary data $\alpha$, generalizing the $P_D(Z)$ of \cite{ks}.  Denote by $Q$ a Klingen parabolic of $\GSp_4$, i.e.\ a maximal parabolic stabilizing a line in the defining four-dimensional representation of $\GSp_4$.  For a cusp form $\phi$ in the space of an automorphic cuspidal representation $\pi$ of $\GSp_4$ with trivial central character, we define
\[I(\phi,s) = \int_{Z(\mathbf{A})\GSp_4(\mathbf{Q}) \backslash \GSp_4(\mathbf{A})}{E(g,s,\Phi)P_D^\alpha(g)\phi(g)\,dg},\]
where $E(g,s,\Phi)$ is a Klingen Eisenstein series. We also define a second integral $J(\phi,s)$ that may be thought of as a degenerate version of $I(\phi,s)$.  The definitions of $E(g,s,\Phi)$ and $J(\phi,s)$ are given in Section \ref{subsec:globalintegrals} below.  We unfold these integrals and show that they are equal to the partial Spin $L$-function $L^S(\pi,\Spin,s)$ times an integral $I_S(\phi,s)$ or $J_S(\phi,s)$ over some set of bad places $S$.

It follows from a theorem of Li \cite{li} that every cuspidal automorphic representation on $\GSp_4$ possesses a nonzero Fourier coefficent of maximal rank for the Siegel parabolic.  We combine this with calculations at ramified finite places and the Archimedean place to show that one can always choose $D$ and the other data in the integral so that $I(\phi,s)$ is nonvanishing.  In the special case of a representation attached to a holomorphic Siegel cusp form of level one, we prove a more precise statement than this nonvanishing: for a particular choice of data at the Archimedean place, $I(\phi,s)$ is equal to the completed $L$-function.  This recovers the main result of Kohnen-Skoruppa \cite{ks}.

The integral $I(\phi,s)$ unfolds to a \emph{non-unique model}.  In the majority of the Rankin-Selberg literature, the unfolded integral involves a model of the cuspidal representation $\pi$ that is known to be unique, such as the Whittaker or Bessel model.  This uniqueness allows one to factorize the integral (for suitable data) and check, place-by-place, that the local factor of the unfolded integral is equal to the local Langlands $L$-function.  However, there are a handful of examples (\cite{bh2}, \cite{psr}, \cite{bfg}, \cite{gs}) of Rankin-Selberg convolutions that unfold to a model that is not unique, yet are still known to represent $L$-functions.  The integral in this paper is one more such example.

We are interested in $I(\phi,s)$ and $J(\phi,s)$ for several reasons.  They have the unusual ability to represent the Spin $L$-function for all cuspidal automorphic representations (for suitable data), rather than either only generic or only holomorphic ones.  They are also members of the limited class of Rankin-Selberg integrals that can potentially be reinterpreted in terms of geometric objects on a Shimura variety, and so could prove useful in answering arithmetic questions. There are higher rank generalizations of both these integrals to $\GSp_6$ that again give the Spin $L$-function \cite{pollackShah}.  The integrals in \cite{pollackShah} should have connections to motivic questions.

We briefly describe the special functions $P_D^\alpha$.  Recall that there is an isomorphism between $\PGSp_4$ and a split special orthogonal group $\SO(V_5,q)$.  Here $V_5$ is a certain $5$-dimensional rational vector space and $q$ is a quadratic form on $V_5$.  Let $D$ be nonzero and square-free, and choose $v_D$ in $V_5(\mathbf Q)$ such that $q(v_D) = D$.  Let $\alpha = \prod_v{\alpha_v}$ be a factorizable function in $\mathcal S(V_5(\mathbf A))$, the space of Schwartz functions on $V_5(\mathbf A)$, equal to the characteristic function of $V_5(\mathbf Z_p)$ almost everywhere.  In order to make the calculations cleaner, we assume that $\GSp_4$ acts on $V_5$ on the right.  The special automorphic function $P_D^\alpha$ is
\[P_D^\alpha(g) = \sum_{\delta \in \operatorname{Stab}(v_D)(\mathbf{Q})\backslash \GSp_4(\mathbf{Q})}{\alpha(v_D\delta g)},\]
where $\operatorname{Stab}(v_D)$ is the stabilizer of $v_D$ in $\GSp_4$.  (See Section \ref{P_D} for more details.  When $D\equiv 1 \pmod{4}$, we will make slightly different choices to avoid issues at the prime 2.)  Actually, the condition that $\alpha_\infty$ be Schwartz is too restrictive; to treat holomorphic Siegel modular forms it will behoove us to take $\alpha_\infty$ decaying polynomially at infinity.

One important difference between this paper and \cite{ks} is in our definition of the special function $P_D^\alpha$.  In \cite{ks}, $P_D$ is defined as the lift to the Maass subspace of a certain Poincare series for Jacobi forms.  The authors then use Hecke operators on Jacobi forms and the special properties of the Fourier-Jacobi expansion of elements in the Maass subspace as their key tools to relate $I(\phi,s)$ to the Spin $L$-function.  However, it is not all that difficult to check that the $P_D$ of \cite{ks} may be expressed as a sum of the above the type; hence our definition of $P_D^\alpha$.  By expressing $P_D^\alpha$ as a sum, we can use it to help unfold the global integral.  Thus our unfolding is different from that in \cite{ks}, as is our proof that the unfolded integral represents the Spin $L$-function.

When $D$ is negative, the data $\alpha$ may be chosen so that $\ol{P_D^\alpha(g)}$ corresponds to a classical Siegel modular form $\ol{P_D^\alpha(Z)}$, where $\ol{\cdot}$ denotes complex conjugation.  In this case, $\ol{P_D^\alpha(Z)}$ has an expression as a sum
\begin{equation}\label{PDintro}\overline{P_D^\alpha(Z)}= \sum_{v \in V_5(\mathbf{Z}), q(v) = -|D|}{\frac{1}{Q_v(Z)^r}}.\end{equation}
Here $Z$ is in the Siegel upper half-space, $Q_v$ is a certain quadratic form on the space of complex two-by-two symmetric matrices, and $r \geq 6$ is an integer.  (See Section \ref{P_D hol}.)  Modular forms similar to (\ref{PDintro}) have been considered previously by many authors in different contexts.  Zagier \cite{zagier} appears to have first defined the analogous functions for modular forms ($\PGL_2 \simeq \SO_3$) and Hilbert modular forms over a real quadratic field ($\PGL^*_{2,\mathbf{Q}(\sqrt{D})} \simeq \SO_4$).

Finally, let us mention some other integrals for the Spin $L$-function on symplectic groups. There are integrals of Novodvorsky \cite{nov} and Piatetski-Shapiro \cite{ps} (following Andrianov \cite{a4}) on $\GSp_4$, which unfold to the Whittaker and Bessel models, respectively.  There are also integrals of Bump-Ginzburg \cite{bg} on $\GSp_6, \GSp_8$ and $\GSp_{10}$, which unfold to the Whittaker models on these groups.  The integral in this paper draws upon similar techniques to \cite{pollackShah}, in which we give a new Rankin-Selberg integral for the Spin $L$-function on $\GSp_6$. Although the integral in \cite{pollackShah} does not apply to holomorphic Siegel modular forms, it uses the same special function $P_D^\alpha$ on $\GSp_4$ used here, and also unfolds to a non-unique model. 

The contents of the various sections are as follows.  In Section \ref{grps etc}, we define the groups, the global Rankin-Selberg integrals, and the special functions that we use.  In Section \ref{KS:global}, we unfold the integrals, discuss the general technique of Piatetski-Shapiro and Rallis to analyze non-unique models, and state the main theorems.  In Section \ref{unramified} we give the unramified calculation, and in Section \ref{ramified} we control the behavior at ramified finite primes and the Archimedean place.  In Section \ref{hol} we choose data appropriate for Siegel modular forms and precisely calculate a corresponding Archimedean integral $I_\infty(\phi,s)$ in terms of $\Gamma$-functions.

$\textrm{ }$

\noindent \emph{Acknowledgments}:  We thank Christopher Skinner for bringing to our attention the paper \cite{ks} and also for many helpful conversations during the course of this research.  We would also like to thank the anonymous referee for many helpful suggestions that substantially improved the exposition of this paper.

\section{Global constructions}\label{grps etc} In this section we define the global Rankin-Selberg integrals considered in this paper.
\subsection{Groups and embeddings} \label{subsec:groupdefs} We define
\[J_4=\left(\begin{array}{cc} & \mathbf{1}_2 \\ -\mathbf{1}_2 & \end{array}\right)\]
and $\GSp_4 = \{g \in \GL_4| gJ_4{}^{t}g = \nu(g) J_4\},$ where $\nu: \GSp_4 \rightarrow \GL_1$ is the similitude and $\mathbf{1}_n$ denotes the $n\times n$ identity matrix.  The letter $Z$ is used to denote the diagonal center of $\GSp_4$.  Denote by $W_4 = \Q^4$ the defining $4$-dimensional representation of $\GSp_4$ and by $\langle \cdot , \cdot\rangle$ the symplectic pairing.  We let $\GSp_4$ act on $W_4$ on the right.  That is, we think of $W_4$ as row vectors, and the action of $\GSp_4$ is the usual right action of matrices on row vectors.  Denote the ordered basis of $W_4$ by $\{e_1,e_2,f_1,f_2\}$, so that $\langle e_i, f_j \rangle = \delta_{i,j}$.    

Let $L = \mathbf{Q}[x]/(x^2-D)$ be a quadratic etale $\mathbf{Q}$-algebra, where $D$ is square-free and nonzero.  Let $\GL_{2,L}$ denote the restriction of scalars $\mathrm{Res}^L_{\mathbf{Q}} \GL_2$, and denote by $\GL_{2,L}^*$ the subgroup of $\GL_{2,L}$ with determinant in $\mathbf{G}_m \subseteq \mathrm{Res}_\mathbf{Q}^L \mathbf{G}_m$.  (In particular, $\GL_{2,L}^*(\mathbf{Q})$ is the set of invertible $2 \times 2$ matrices with entries in $L$ and determinant in $\mathbf{Q}^\times$.)  Suppose first that $D \not \equiv 1 \pmod{4}$.  One can define an embedding $\GL_{2,L}^* \rightarrow \GSp_4$ as follows.  We consider the $L$-vector space $L^2$ and write $c_1$ and $c_2$ for its two standard basis elements.  There is an $L$-valued symplectic form $\langle \cdot , \cdot \rangle_L$ on $L^2$ given by
\[\langle u_1 c_1 + u_2 c_2 , v_1 c_1 + v_2 c_2 \rangle_L = u_1 v_2 - u_2 v_1.\]
Suppose that $\alpha$ denotes the image of $x$ in $L = \Q[x]/(x^2-D)$, so that $\alpha$ is a square root of $D$.  We write $\ol{\cdot}$ for the action of the nontrivial Galois automorphism on $L$.  Define a $\Q$-valued symplectic form $\langle \cdot,\cdot \rangle'$ on $L^2$ by $\langle u, v \rangle' = \tr_{L/\Q}\left(\alpha \langle u, v\rangle_L\right)/(2D)$.    Consider the $\Q$-basis
\begin{equation}\label{basis} e_1 = c_1, \; e_2 = \alpha c_1, \; f_1 = -\ol{\alpha} c_2, \; f_2 = c_2\end{equation}
of $L^2$.  Then $\langle e_i, f_j \rangle' = \delta_{ij}$, and $\langle e_i, e_j \rangle' = \langle f_i, f_j \rangle' = 0$.  Hence the basis (\ref{basis}) yields an identification
\begin{equation}\label{pairs} (L^2, \langle \cdot , \cdot \rangle') \simeq (W_4, \langle \cdot, \cdot \rangle).\end{equation}
If $g \in \GL_{2,L}$ and $u, v \in L^2$, then $\langle ug, vg \rangle_L = \det(g) \langle u, v \rangle_L$.  Hence if $g \in \GL_{2,L}^*$, then $\langle ug, vg \rangle' = \det(g) \langle u, v \rangle'$.  Thus (\ref{pairs}) yields the inclusion $\GL_{2,L}^* \rightarrow \GSp_4$.

Suppose that $g = \mm{u_1}{u_2}{u_3}{u_4} \in \GL_{2,L}^*$, where $u_i = \epsilon_i + \alpha \eta_i$ for $\epsilon_i, \eta_i$ in $\Q$.  Then under the mapping $\GL_{2,L}^* \rightarrow \GSp_4$, the image of $g$ is the matrix
\begin{equation}\label{GGmatrix2} \left(\begin{array}{cc|cc} \epsilon_1 & \eta_1 & \eta_2 & \epsilon_2 \\ D \eta_1 & \epsilon_1 & \epsilon_2 & D \eta_2 \\ \hline D \eta_3 & \epsilon_3 & \epsilon_4 & D \eta_4 \\ \epsilon_3 & \eta_3 & \eta_4 & \epsilon_4 \end{array}\right). \end{equation}
If $D \equiv 1 \pmod{4}$ and one uses the minimal polynomial $x^2+x+\frac{1-D}{4}$ of $\frac{-1+\sqrt{D}}{2}$ in place of $x^2-D$ and adjusts the definition of the pairing to $\langle u, v \rangle' = \tr_{L/\Q}\left((2 \alpha + 1) \langle u, v\rangle_L\right)/D$, the preceding discussion (with $\alpha$ the image of $x$) leads in exactly the same way to the mapping of $g=\mm{u_1}{u_2}{u_3}{u_4}$ to
\begin{equation}\label{GGmatrix3} \left(\begin{array}{cc|cc}\epsilon_1&\eta_1 &\eta_2&\epsilon_2-\eta_2 \\ \frac{D-1}{4}\eta_1&\epsilon_1-\eta_1&\epsilon_2-\eta_2&-\epsilon_2+\eta_2\frac{D+3}{4}\\ \hline \epsilon_3+\frac{D-1}{4}\eta_3&\epsilon_3&\epsilon_4&\frac{D-1}{4}\eta_4 \\ \epsilon_3 & \eta_3 & \eta_4& \epsilon_4-\eta_4\end{array}\right).\end{equation}
We remark that the image of $\GL_{2,L}^*$ has an alternative description as the centralizer of the matrix $\mm{\epsilon_D}{}{}{^t\epsilon_D}$ inside $\GSp_4$, where we take $\epsilon_D = \mm{}{1}{D}{}$ if $D \equiv 2, 3 \pmod{4}$ and $\epsilon_D = \mm{}{1}{\frac{D-1}{4}}{-1}$ if $D \equiv 1 \pmod{4}$.

If $D =1$, so that $L \cong \Q \times \Q$, then $\GL_{2,L}^*$ is isomorphic to the subgroup of $\GL_2 \times \GL_2$ consisting of the $(g_1, g_2)$ with $\det(g_1) = \det(g_2)$.  For later use, we record the following explicit realization of this identity.  Set $e_1' = e_1+e_2$, $e_2' = e_2$, $f_1' = f_1$, and $f_2' = f_1 - f_2$.  Then if $D=1$, the matrix in (\ref{GGmatrix3}) sends
\begin{align}\label{D=1} e_1' &\mapsto \epsilon_1e_1' + \epsilon_2f_1' \\ f_1' &\mapsto \epsilon_3e_1'+\epsilon_4f_1' \nonumber \\ e_2' &\mapsto (\epsilon_1 -\eta_1) e_2' + (\eta_2 - \epsilon_2) f_2' \nonumber \\ f_2' &\mapsto (\eta_3 - \epsilon_3) e_2' + (\epsilon_4 - \eta_4) f_2'. \nonumber \end{align}

\subsection{The special function $P_D^\alpha$}\label{P_D}
We now define the special function $P_D^\alpha$ to be used in the integral representations.  Recall that $\nu$ denotes the similitude character.  Define 
\[\wedge^2_0(W_4) = \ker\{\wedge^2 W_4 \stackrel{v\wedge w \mapsto \langle v,w\rangle}{\longrightarrow} 1 \otimes \nu\}\]
and define a 5-dimensional representation $V_5$ of $\GSp_4$ as the diagonal action on $V_5 = \wedge^2_0(W_4) \otimes \nu^{-1}$.  As $\wedge^4 W_4 \otimes \nu^{-2}$ is the trivial representation of $\GSp_4$, the exterior product $\wedge: V_5 \otimes V_5 \rightarrow \wedge^4 W_4 \otimes \nu^{-2}$ defines the invariant quadratic form $q$ on $V_5$.  More precisely, we first define $( \cdot, \cdot )$ via
\[v \wedge w = (v,w) e_1 \wedge e_2 \wedge f_1 \wedge f_2,\]
where $e_1 \wedge e_2 \wedge f_1 \wedge f_2$ is here considered as an element of $\wedge^4 W_4 \otimes \nu^{-2}$.  The triviality of $\wedge^4 W_4 \otimes \nu^{-2}$ shows that $(\cdot,\cdot)$ is invariant.  We then define the quadratic form $q(v)=\frac{1}{2}(v,v)$.  If $D \not\equiv 1 \pmod{4}$, we define an integral structure on $V_5$ by choosing the integral basis
\[\set{e_1\wedge f_1 - e_2\wedge f_2, e_1 \wedge e_2, e_1 \wedge f_2, e_2 \wedge f_1, f_1\wedge f_2}\]
and define
\begin{equation} \label{vdnot1mod4} v_D = De_1 \wedge f_2 + e_2 \wedge f_1. \end{equation}
Then $q(v_D) = (v_D,v_D) = D$, and one can check that the stabilizer of $v_D$ is ${\GL_{2,L}^*}$ in the embedding described by (\ref{GGmatrix2}).  If $D \equiv 1 \pmod{4}$, we replace $e_1\wedge f_1 - e_2\wedge f_2$ by $\frac{1}{2}(e_1\wedge f_1 - e_2\wedge f_2)$ in the integral basis and set
\begin{equation} \label{vd1mod4} v_D = \frac{D-1}{4}e_1 \wedge f_2+\frac{1}{2}(e_1\wedge f_1 - e_2\wedge f_2)+ e_2 \wedge f_1\end{equation}
so that $q(D)=\frac{D}{4}$.  Once again, the stabilizer is $\GL_{2,L}^*$ in its embedding from (\ref{GGmatrix3}).  To see that the chosen integral structure is $\GSp_4(\mathbf{Z})$-stable, we observe that since $e_1\wedge f_1 + e_2\wedge f_2$ is fixed by the action of $g$ (due to the twisting by $\nu^{-1}$ in the definition of $V_5$), we have
\begin{align*}
	\frac{1}{2}(e_1\wedge f_1 - e_2\wedge f_2)g &= \frac{1}{2}(e_1\wedge f_1 + e_2\wedge f_2-2e_2 \wedge f_2)g = \frac{1}{2}(e_1\wedge f_1 + e_2\wedge f_2)-(e_2\wedge f_2)g\\
		&=\frac{1}{2}(e_1\wedge f_1 - e_2\wedge f_2)+e_2 \wedge f_2 - (e_2\wedge f_2)g,
\end{align*}
which is in $V_5(\mathbf{Z})$.

Now take $\alpha = \prod_{v}{\alpha_v}$ to be factorizable in the Schwartz space $\mathcal{S}(V_5(\mathbf{A}))$.  For all but finitely many finite primes $p$, we require that $\alpha_p$ is the characteristic function of $V_5(\mathbf{Z}_p)$ (using the integral structure defined above).  We define
\[P_D^\alpha(g) = \sum_{\delta \in \GL_{2,L}^*(\mathbf{Q})\backslash \GSp_4(\mathbf{Q})}{\alpha(v_D\delta g)}.\]
In Section \ref{P_D hol} we will choose an $\alpha_\infty$ that is very convenient for computing with holomorphic Siegel modular forms, although this $\alpha_\infty$ will not be a Schwartz function.

\subsection{The global integrals} \label{subsec:globalintegrals}
We now define the two global Rankin-Selberg convolutions to be studied in this paper.  The integral representations use the Klingen Eisenstein series on $\GSp_4$.  We define the Klingen parabolic, denoted $Q$, to be the stabilizer of the line $\mathbf{Q} f_2$ via the right action.  Alternatively, $Q$ is the subgroup of $\GSp_4$ consisting of elements of the form
\begin{equation*} g= \left(\begin{array}{cccc} * & 0 & * & * \\ * & a & * & * \\ * & 0 & * & * \\ 0 & 0 & 0 & d \end{array}\right).\end{equation*}
For such a $g$, $\nu(g) = ad$, $f_2 g = d f_2$, and $\delta_Q(g) = |\frac{a}{d}|^2$.

If $\Phi = \prod_{v}{\Phi_v} \in \mathcal{S}(W_4(\mathbf{A}))$ is a factorizable Schwartz-Bruhat function, define
\begin{equation}\label{f_pdef}f_v^\Phi(g,s) = |\nu(g)|_v^{2s} \int_{\GL_1(\mathbf{\Q}_v)}{\Phi_v(tf_2g)|t|_v^{4s}\,dt}\end{equation}
and
\begin{equation}\label{fPhi}f^\Phi(g,s) = |\nu(g)|^{2s} \int_{\GL_1(\mathbf{A})}{\Phi(tf_2g)|t|^{4s}\,dt}= \prod_v{f^\Phi_v(g_v,s)}. \end{equation}
Then $f^\Phi(g,s) \in \ind_{Q(\A)}^{\GSp_4(\A)}(\delta_Q^s)$, where $\delta_Q$ is the modulus character of the Klingen parabolic.   The Eisenstein series is
\[E(g,s,\Phi) = \sum_{\gamma \in Q(\mathbf{Q})\backslash\GSp_4(\mathbf{Q})}{f^\Phi(\gamma g,s)}.\]
\begin{remark} If $\Phi_p$ is the characteristic function of $W_4(\Z_p)$ for $p < \infty$ and $\Phi_\infty(v)$ is the Gaussian $e^{-\pi ||v||^2}$, then $E(g,s, \Phi)$ is the normalized spherical Eisenstein series.  Indeed, denote by $f_v^0(g,s)$ the unique element of $\ind_{Q(\Q_v)}^{\GSp_4(\Q_v)}(\delta_Q^s)$ that satisfies $f_v^0(k,s) = 1$ for all $k \in K_v$.  Then for every $v$, with this choice of $\Phi_v$, $f_v^\Phi(k,s) = \zeta_v(4s)$ for $k\in K_v$, where $\zeta_\R(s) = \Gamma_\R(s) = \pi^{-s/2}\Gamma(s/2)$.  Hence by the Iwasawa decomposition, $f^\Phi_v(g,s) = \zeta_v(4s) f_v^0(g,s)$.  Langlands' theorem of the constant term and functional equation then imply, by an intertwining operator calculation, that this $E(g,s, \Phi)$ has finitely many poles and satisfies the functional equation $E(g,s,\Phi) = E(g,1-s,\Phi)$.  For general Schwartz-Bruhat functions $\Phi$, one can obtain a bound on the poles of $E(g,s,\Phi)$ and its functional equation by using Poisson summation. Since we will not use these facts, we omit the proof.\end{remark}

Suppose $\pi$ is a cuspidal automorphic representation on $\GSp_4$ with trivial central character and $\phi$ is a cusp form in the space of $\pi$.  We define our global integrals as follows:
\begin{equation} \label{eqn:intpd} I(\phi,s) = \int_{Z(\mathbf{A})\GSp_4(\mathbf{Q}) \backslash \GSp_4(\mathbf{A})}{E(g,s,\Phi)P_D^\alpha(g)\phi(g)\,dg}\end{equation}
and
\begin{equation} \label{eqn:intgl2star} J(\phi,s) = \int_{Z(\mathbf{A}){\GL_{2,L}^*}(\mathbf{Q}) \backslash {\GL_{2,L}^*}(\mathbf{A})}{E(g,s,\Phi)\phi(g)\,dg}.\end{equation}
The integrals $I(\phi,s)$ and $J(\phi,s)$ depend on $D$ and $\Phi$, but we omit this data from the notation.  The main theorems of this paper, spelled out in the next section, relate both of these integrals to the partial Spin $L$-function on $\GSp_4$, $L^S(\pi,\Spin,2s-\frac{1}{2})$.

The integrals $I(\phi,s)$ and $J(\phi,s)$ are very closely related. In the proof of Proposition \ref{unfoldProp} below, after unfolding the sum defining $P_D^\alpha$ in $I(\phi,s)$, we find that the integral $J(\phi,s)$ appears as an inner integral of this partially unfolded $I(\phi,s)$.  Note also that $J(\phi,s)$ is similar to the integrals of Andrianov \cite{a4} and Piatetski-Shapiro \cite{ps} for the Spin $L$-function on $\GSp_4$, both of which unfold to the Bessel model.  The main difference is that the integrals of \cite{a4} and \cite{ps} use an Eisenstein series for the group $\GL_{2,L}^*$ as opposed to one for $\GSp_4$.  The benefit of studying both integrals is that the integral $I(\phi,s)$ provides a more flexible construction than $J(\phi,s)$, while $J(\phi,s)$ makes the connection to the geometry of the Shimura variety (and with this the potential to answer motivic questions) more direct.
%=================================================
\section{Main theorems}\label{KS:global}
%=================================================
In this section we unfold the global integrals and state the main theorems of the paper.  The calculations proving these theorems are given in Sections \ref{unramified}, \ref{ramified}, and \ref{hol}.  As we mentioned above, the integral unfolds to a model that is not unique.  The general strategy for analyzing such integrals -- which we follow -- was developed by Piatetski-Shapiro and Rallis \cite{psr}, and is somewhat well-known \cite{bfg, gs}.  For the convenience of the reader, we briefly review this strategy.
\subsection{Unfolding}
Define the standard Siegel parabolic $P$ on $\GSp_4$ to be the stabilizer of the isotropic subspace of $W_4$ spanned by $f_1, f_2$.  Denote by $U$ its unipotent radical, and by $M$ the Levi consisting of elements of $\GSp_4$ whose lower left and upper right two-by-two blocks are zero.  We denote by $N$ the unipotent radical of the upper triangular Borel in $\GL_{2,L}^*$.  Then, under the embedding $\GL_{2,L}^* \rightarrow \GSp_4$, $N$ is contained in $U$.  This can be seen immediately from the matrices (\ref{GGmatrix2}) and (\ref{GGmatrix3}). 

In fact, define a unitary character $\chi: U(\mathbf{A}) \rightarrow \mathbf{C}^\times$ as follows.  As is usual, pick once and for all an additive character $\psi: \mathbf{Q}\backslash \mathbf{A} \rightarrow \mathbf{C}^\times$, with conductor equal to one at all finite places, and $\psi_\infty(x) = e^{2 \pi i x}$ for $x \in \R$.  If
\[u = \left(\begin{array}{cc|cc}1& &u_{11}&u_{12} \\ & 1&u_{12}&u_{22}\\ \hline & &1& \\ & & &1\end{array}\right),\]
then set
\[\chi(u) = \begin{cases} \psi(-Du_{11} + u_{22}) & D \not \equiv 1 \pmod{4} \\ \psi(\frac{1-D}{4}u_{11} + u_{12} + u_{22}) & D \equiv 1 \pmod{4} \end{cases}.\]
Note that $\chi$ is trivial on $N$.  With this definition of $\chi$, we define functions $\phi_\chi$ and $\alpha_\chi$ on $\GSp_4$ by the absolutely convergent integrals
\begin{equation} \label{eqn:phialphachidef} \phi_\chi(g) = \int_{U(\mathbf{Q})\backslash U(\mathbf{A})}{\chi^{-1}(u)\phi(ug) \, du} \quad \text{and}\quad \alpha_\chi(g) = \int_{N(\mathbf{A})\backslash U(\mathbf{A})}{\chi(u)\alpha(v_Dug) \, du}.\end{equation}
The second integral is factorizable since $\chi = \prod_p \chi_p$ and $\alpha$ are, so we may write $\alpha_\chi=\prod_p \alpha_{\chi,p}$, where
\begin{equation} \label{eqn:alphachivdef} \alpha_{\chi,p}(g_p) = \int_{N(\mathbf{Q}_p)\backslash U(\mathbf{Q}_p)}\chi_p(u)\alpha_p(v_{D,p}ug_p) \, du.\end{equation}
Here $p$ denotes either an infinite or finite place of $\mathbf{Q}$.

\begin{lemma}
The integral for $\alpha_\chi$ in (\ref{eqn:phialphachidef}) is absolutely convergent.
\end{lemma}

\begin{proof}
Note that the function $\alpha': V_5(\mathbf{A})\rightarrow \mathbf{C}$ defined by $\alpha'(v) = \alpha(vg)$ is still a Schwartz-Bruhat function.  Set
\begin{equation*} n_{24}(z) = \left(\begin{array}{cc|cc} 1& & &\\ &1& &z\\ \hline & &1& \\ & & &1\end{array}\right).\end{equation*}
Then $v_D n_{24}(z) = v_D + z f_2 \wedge f_1$ for all $D$.  We obtain
\begin{align}\label{vDz} \alpha_{\chi}(g) &= \int_{N(\A) \backslash U(\A)}{\chi(u) \alpha(v_D u g)\, du} \nonumber \\ &= \int_{\mathbf{G}_a(\A)}{\psi(z) \alpha'(v_Dn_{24}(z))\, dz} \nonumber \\ &= \int_{\mathbf{G}_a(\A)}{\psi(z) \alpha'(v_D + z f_2 \wedge f_1)\, dz}. \end{align}
Since $\alpha'$ is a Schwartz-Bruhat function, it is easily seen that the function $z \mapsto \alpha'(v_D + z f_2 \wedge f_1)$ is a Schwartz-Bruhat function on $\A$.  Hence the integral (\ref{vDz}) is absolutely convergent.
\end{proof}

\begin{proposition}\label{unfoldProp}  The global integrals $I$ and $J$ unfold as
\begin{equation}\label{unfoldI}I(\phi,s) = \int_{U(\mathbf{A})Z(\mathbf{A})\backslash\GSp_4(\mathbf{A})}{f^\Phi(g,s)\phi_\chi(g)\alpha_\chi(g)\,dg}\end{equation}
and
\begin{equation}\label{unfoldJ}J(\phi,s) = \int_{N(\mathbf{A})Z(\mathbf{A})\backslash\GL_{2,L}^*(\mathbf{A})}{f^\Phi(g,s)\phi_\chi(g) \,dg}.\end{equation}
\end{proposition}
\begin{proof} We first unfold $I(\phi,s)$.  Assume for now that $L$ is a field.  Unfolding $P_D^\alpha$, one obtains
\begin{equation}\label{unfold0} I(\phi,s) = \int_{\GL_{2,L}^*(\Q)Z(\mathbf{A})\backslash\GSp_4(\mathbf{A})}{E(g,s,\Phi)\phi(g)\alpha(v_Dg)\,dg}. \end{equation}
Although it is not needed below, note that the integral $J(\phi,s)$ appears if one performs an inner integration in (\ref{unfold0}) over the group $\GL_{2,L}^*(\A)$, since $v_D$ is stabilized by $\GL_{2,L}^*(\A)$.

Next, we unfold the Eisenstein series.  Set $B'(\Q) = Q(\Q) \cap \GL_{2,L}^*(\Q)$.  We claim that the double coset space $Q(\Q) \backslash \GSp_4(\Q) \slash \GL_{2,L}^*(\Q)$ is a singleton, and thus
\begin{equation}\label{Eis_Sum} E(g,s,\Phi) = \sum_{\gamma \in B'(\Q) \backslash \GL_{2,L}^*(\Q)}{f^\Phi(\gamma g, s)}. \end{equation}
Assuming this fact for the moment, one obtains
\begin{equation}\label{unfold1}I(\phi,s) = \int_{B'(\mathbf{Q})Z(\mathbf{A}) \backslash \GSp_4(\mathbf{A})}{\alpha(v_Dg)f^\Phi(g,s)\phi(g)\,dg}\end{equation}
by inserting (\ref{Eis_Sum}) into (\ref{unfold0}) and unfolding the sum over $\GL_{2,L}^*(\Q)$.  To check the claim, note that the map $Q(\Q)g \mapsto \Q f_2 g$ defines a bijection between $Q(\Q) \backslash \GSp_4(\Q)$ and lines in $W_4$.  Moreover, ${\GL_{2,L}^*}$ acts transitively on the lines in $W_4$ since $\SL_{2,L}$ acts transitively on the nonzero elements of its two-dimensional representation $L^2$.  Hence $Q(\Q) \backslash \GSp_4(\Q) \slash \GL_{2,L}^*(\Q)$ is indeed a singleton.

Using the notation (\ref{GGmatrix2}) or (\ref{GGmatrix3}), $B'$ consists of the elements of $\GSp_4$ with $\epsilon_3 = \eta_1 = \eta_3 = \eta_4 =0$.  Now
\begin{equation}\label{unfold2}\int_{N(\mathbf{Q})\backslash N(\mathbf{A})}\phi(ng) \, dn = \sum_{\gamma \in \GL_1(\Q)}{\phi_\chi(\gamma g)},\end{equation}
where $\GL_1(\Q)$ denotes the matrices $\mm{\lambda}{}{}{1}$, $\lambda \in \Q^\times$.  Integrating the right-hand side of (\ref{unfold1}) over $N(\Q) \backslash N(\A)$, and then applying the Fourier expansion (\ref{unfold2}) into (\ref{unfold1}), we obtain
\[\int_{N(\A)Z(\A)\backslash \GSp_4(\A)}{f^\Phi(g,s) \phi_\chi(g) \alpha(v_D g)\,dg}.\]
The proposition follows in this case.

Now suppose that $D = 1$, so that $L \cong \mathbf{Q} \times \mathbf{Q}$ is split.  Then ${\GL_{2,L}^*}$ acts on the lines in $W_4$ with three orbits, represented by $\mathbf{Q} f_2, \mathbf{Q}(f_1'),$ and $\mathbf{Q}(f_2')$.  (Recall the notation (\ref{D=1}).)  Indeed, consider the action of $\GL_{2,L}^*$ on $W_4$ in terms of the basis (\ref{D=1}).  The action described in (\ref{D=1}) shows that $\mathbf{Q}f_1'$ and $\mathbf{Q}f_2'$ represent the orbits of lines in $\mathrm{span}(e_1',f_1')$ and $\mathrm{span}(e_2',f_2')$ respectively, while the fact that $\SL_2(\Q)$ acts transitively on $\Q^2\setminus \{0\}$ shows that $\Q (f_1'- f_2') = \Q f_2$ represents an orbit containing all other lines in $W_4$.  Fix $\gamma_{1}, \gamma_{2} \in \GSp_4(\Q)$ satisfying $f_2 \gamma_1 = f_1'$, $f_2 \gamma_2 = f_2'$.  We have just computed that $Q(\Q) \backslash \GSp_4(\Q) \slash \GL_{2,L}^* = \{1, \gamma_1, \gamma_2\}$.  Denote by $B_i'$ the subgroup of $\GL_{2,L}^*$ that stabilizes the line $\Q f_i'$.

Using these double coset representatives to unfold (\ref{unfold0}), we obtain
\[I(\phi,s) = I_{f_2}(\phi,s) + I_{f_1'}(\phi,s) + I_{f_2'}(\phi,s),\]
where $I_{f_2}(\phi,s)$ is the integral on the right-hand side of (\ref{unfold1}), and 
\[I_{f_i'}(\phi,s) = \int_{B_i'(\mathbf{Q})Z(\mathbf{A}) \backslash \GSp_4(\mathbf{A})}{\alpha(v_Dg)f^\Phi(\gamma_i g,s)\phi(g)\,dg}.\]

The unfolding of the integral $I_{f_2}(\phi,s)$ proceeds as above.  We check that the integrals $I_{f_i'}(\phi,s)$ vanish by the cuspidality of $\phi$.  Assume $i =1$; the case $i=2$ is identical.  Using the ordered symplectic basis $e_1', e_2', f_2', f_1'$ of $W_4$, $B_1'$ is the subgroup of $\GSp_4$ of matrices with the form
\begin{equation}\label{B_1'Shape} \left( \begin{array}{cccc} *&0&0&* \\ 0&*&*&0 \\ 0&*&*&0 \\ 0&0&0&* \end{array}\right). \end{equation}
Denote by $U_0$ the one-dimensional unipotent subgroup
\begin{equation*} \left( \begin{array}{cccc} 1&0&0&* \\ 0&1&0&0 \\ 0&0&1&0 \\ 0&0&0&1 \end{array}\right)\end{equation*}
of $\GSp_4$.  Let $\overline{U}$ be the abelian group $U/U_0$, where $U$ is the unipotent subgroup of $\GSp_4$ consisting of the matrices
\begin{equation*} u(x_1,x_2) \triangleq \left( \begin{array}{cccc} 1&x_1&x_2&* \\ 0&1&0&x_2 \\ 0&0&1&-x_1 \\ 0&0&0&1 \end{array}\right). \end{equation*}
Finally, let $B_1'' \subseteq B_1'$ denote the matrices in $\GSp_4$ that are of the form
\begin{equation*} \left( \begin{array}{cccc} a&0&0&* \\ 0&a&*&0 \\ 0&0&d&0 \\ 0&0&0&d \end{array}\right). \end{equation*}
We integrate $\phi$ over $U_0(\mathbf{Q}) \backslash U_0(\mathbf{A})$ and Fourier expand along $\overline{U}$ to obtain the identity
\begin{equation}\label{FE split} \phi_{U_0}(g) \triangleq \int_{U_0(\Q)\backslash U_0(\A)}{\phi(ug)\,du} = \sum_{\gamma \in B_1''(\Q) \backslash B_1'(\Q)} {\phi_{(1,0,0)}(\gamma g)}, \end{equation}
where
\[\phi_{(1,0,0)}(g) \triangleq \int_{\overline{U}(\Q)\backslash \overline{U}(\A)}{\psi(x_1) \phi_{U_0}(u(x_1,x_2) g)\, du}.\]

Next, we perform an inner integration over $U_0(\Q)\backslash U_0(\A)$ in $I_{f_1'}(\phi,s)$ and apply the Fourier expansion (\ref{FE split}) to obtain
\begin{equation}\label{splitUnfold1} I_{f_1'}(\phi,s) = \int_{B_1''(\mathbf{Q})U_0(\A)Z(\mathbf{A}) \backslash \GSp_4(\mathbf{A})}{\alpha(v_Dg)f^{\Phi}(\gamma_1 g,s)\phi_{(1,0,0)}(g)\,dg}.\end{equation}
Denote by $N_0$ the one-dimensional unipotent subgroup of $\GSp_4$ consisting of the matrices
\begin{equation*} \left( \begin{array}{cccc} 1&0&0&0 \\ 0&1&*&0 \\ 0&0&1&0 \\ 0&0&0&1 \end{array}\right).\end{equation*}
We perform an inner integration in (\ref{splitUnfold1}) over the group $N_0(\Q) \backslash N_0(\A)$, which shows that $I_{f_1'}(\phi,s)$ vanishes by the cuspidality of $\phi$ along the unipotent group
\begin{equation*} \left( \begin{array}{cccc} 1&0&*&* \\ 0&1&*&* \\ 0&0&1&0 \\ 0&0&0&1 \end{array}\right).\end{equation*}

Now we unfold $J(\phi,s)$.  The unfolding proceeds identically to the unfolding of $I(\phi,s)$ starting from the step (\ref{unfold0}).  For example, assume $L$ is a field.  The coset and stabilizer computations for $Q(\Q) \backslash \GSp_4(\Q) \slash \GL_{2,L}^*(\Q)$ above show that
\begin{equation}\label{Junfold2} J(\phi,s) = \int_{B'(\Q)Z(\A) \backslash \GSp_4(\A)}{f^\Phi(g,s)\phi(g)\,dg}. \end{equation}
We integrate over $N(\Q) \backslash N(\A)$ and use the Fourier expansion (\ref{unfold2}) to obtain
\begin{equation*} J(\phi,s) = \int_{N(\A)Z(\A) \backslash \GL_{2,L}^*(\A)}{f^\Phi(g,s)\phi_\chi(g)\,dg}.\end{equation*}

If instead $D=1$, we first use the coset computation $Q(\Q) \backslash \GSp_4(\Q) \slash \GL_{2,L}^*(\Q) = \{1, \gamma_1, \gamma_2\}$ above to obtain
\begin{equation*} J(\phi,s) = J_{f_2}(\phi,s) + J_{f_1'}(\phi,s) + J_{f_2'}(\phi,s),\end{equation*}
where $J_{f_2}(\phi,s)$ is the integral on the right-hand side of (\ref{Junfold2}), and
\begin{equation*} J_{f_i'}(\phi,s) = \int_{B_i'(\mathbf{Q})Z(\mathbf{A}) \backslash \GL_{2,L}^*(\mathbf{A})}{f^{\Phi}(\gamma_i g,s)\phi(g)\,dg}.\]
The same manipulations used to check the vanishing of $I_{f_i'}(\phi,s)$ show that the integrals $J_{f_1'}(\phi,s)$ and $J_{f_2'}(\phi,s)$ vanish by the cuspidality of $\phi$.  The unfolding of $J_{f,2}(\phi,s)$ is the same as when $L$ is a field, so the proposition follows.
\end{proof}

\subsection{Non-unique models} \label{subsec:nonuniquemodels}
The Fourier coefficient $\phi_\chi$ does not factorize for general cusp forms $\phi$, so it does not follow from the unfolding calculation that the integrals $I$ and $J$ have the structure of an Euler product.  However, Piatetski-Shapiro and Rallis in \cite{psr} introduced a technique that can sometimes factorize the global integral in this situation.

Suppose $(\pi_v,V)$ is an irreducible admissible representation of $\GSp_4(\mathbf{Q}_v)$.  Define a $(U,\chi)$-functional to be a linear map $\ell: V \rightarrow \mathbf{C}$ that satisfies $\ell(\pi(u) v) = \chi(u) \ell(v)$ for all $u \in U(\mathbf{Q}_v)$ and $v \in V$.  If one replaced $U$ in this definition by $U_B$, the unipotent radical of a Borel of $\GSp_4$, and $\chi$ by a nondegenerate character of $U_B$, then one would get the definition of the Whittaker functional.  However, unlike the Whittaker functional, the space of $(U,\chi)$-functionals will usually be infinite dimensional.  Consequently, the Fourier coefficient $\phi_\chi$ will not factorize.  However, the following theorem allows one to recover the Euler product nature of the global integral.

\begin{theorem}\label{KS:unram} Suppose that $p$ is finite, $\pi_p$ is unramified, $\Phi_p$ is the characteristic function of $W_4(\Z_p)$, and $\alpha_p$ is the characteristic function of $V_5(\mathbf{Z}_p)$.  Let $v_0$ denote a nonzero spherical vector in $\pi_p$.  Then, for all $(U,\chi)$-functionals $\ell$,
\[I_{p}(\ell,s) \triangleq \int_{U(\mathbf{Q}_p)Z(\mathbf{Q}_p)\backslash \GSp_4(\mathbf{Q}_p)}{f_p^{\Phi_p}(g,s)\alpha_{\chi,p}(g) \ell(\pi(g)v_0) \,dg} = \ell(v_0)L(\pi_p,\Spin,2s-\frac{1}{2})\]
for $\mathrm{Re}(s)$ sufficiently large.  Similarly,
\[J_{p}(\ell,s) \triangleq \int_{N(\mathbf{Q}_p)Z(\mathbf{Q}_p)\backslash\GL_{2,L}^*(\mathbf{Q}_p)}{f_p^{\Phi_p}(g,s)\ell(\pi(g) v_0) \,dg} = \ell(v_0)L(\pi_p,\Spin,2s-\frac{1}{2})\]
for $\mathrm{Re}(s)$ sufficiently large.
\end{theorem}

We will prove this theorem in the next section.  Note that when $\ell(v_0) = 0$, Theorem \ref{KS:unram} says that the integrals $I_p(\ell,s)$ and $J_p(\ell,s)$ appearing in the theorem are $0$.

For a finite set of places $\Omega$, define
\begin{equation}\label{IOmega} I_{\Omega}(\phi,s) \triangleq \int_{U(\mathbf{A}_\Omega)Z(\mathbf{A}_\Omega)\backslash \GSp_4(\mathbf{A}_\Omega)}{f_\Omega^{\Phi}(g,s)\alpha_{\chi,\Omega}(g) \phi_\chi(g) \,dg},\end{equation}
where for a linear algebraic group $H$, we write $H(\mathbf{A}_\Omega) = \prod_{v \in \Omega}{H(\mathbf{Q}_v)}$, and where $f_\Omega$ and $\alpha_{\chi,\Omega}$ respectively denote the product of the factors of $f^\Phi$ and $\alpha_\chi$ at the places in $\Omega$.  We now apply the argument of the Basic Lemma of \cite[pg. 117]{psr} to deduce the following corollary.  Write $H: \pi \rightarrow \mathbf{C}$ for the functional $\phi \mapsto \phi_\chi(1)$. 
%We also define $\phi_\Omega = \bigotimes_{v\in \Omega} \phi_v$ and use the notation $\Omega^c$ to denote the complement of any set $\Omega$ of places of $\mathbf{Q}$.   
\begin{corollary}\label{partialL} Suppose that the cusp form $\phi$ corresponds to a factorizable element under some choice of isomorphism $\pi \cong \otimes_p \pi_p$.  Suppose that $S$ is a finite set of places of $\Q$ containing all the places excluded in the statement of Theorem \ref{KS:unram}.  That is, if $p \notin S$, then $p$ is finite, $\Phi_p$ is the characteristic function of $W_4(\Z_p)$, $\alpha_p$ is the characteristic function of $V_5(\mathbf{Z}_p)$, and $\pi_p$ is spherical.  Then if $\phi$ is invariant under $\GSp_4(\Z_p)$ for all $p \notin S$, $I(\phi,s) = I_S(\phi, s)L^S(\pi,\Spin,2s-\frac{1}{2})$.\end{corollary}
\begin{proof} The argument we use is well-known: see \cite[Theorem 1.2]{bfg} or \cite[pg.\ 157]{gs}.  For the reader's convenience, we sketch the proof. 

Let $\Omega$ denote a finite set of places of $\Q$ containing $S$.  Then it is the definition of the adelic integral on the right-hand side of (\ref{unfoldI}) that
\[I(\phi,s) = \varinjlim_{\Omega \supseteq S} I_\Omega(\phi,s)\]
in the range of absolute convergence of the integral.  Now, 
\begin{align}\label{BLarg} I_{\Omega \cup \{p\}}(\phi,s) &= \int_{U(\mathbf{A}_{\Omega \cup \{p\}})Z(\mathbf{A}_{\Omega \cup \{p\}})\backslash \GSp_4(\mathbf{A}_{\Omega \cup \{p\}})}{f_{\Omega \cup \{p\}}^{\Phi}(g,s)\alpha_{\chi,{\Omega \cup \{p\}}}(g) \phi_\chi(g) \,dg} \nonumber \\ &= \int_{U(\mathbf{A}_\Omega)Z(\mathbf{A}_\Omega)\backslash \GSp_4(\mathbf{A}_\Omega)}{f_\Omega^{\Phi}(g,s)\alpha_{\chi,\Omega}(g)} \nonumber \\ & \times \left( \int_{U(\Q_p)Z(\Q_p)\backslash \GSp_4(\Q_p)}{f_p^{\Phi_p}(g_p,s)\alpha_{\chi,p}(g_p)\phi_\chi(g_p g)\,dg_p}\right)\,dg \end{align}
since $f^\Phi$ and $\alpha_\chi$ are factorizable.

Fix $g$ in $\GSp_4(\A_\Omega)$ and denote by $V_{\pi_p}$ the space of the representation $\pi_p$.  Since $\phi$ corresponds to a factorizable element of the space of $\pi$ under some choice of isomorphism $\pi \cong \otimes_p \pi_p$ and is $\GSp_4(\mathbf{Z}_p)$-invariant, there is a spherical vector $v_0$ of $V_{\pi_p}$ and a $\pi_p$-equivariant map $\iota_p$ from $V_{\pi_p}$ into the space of cusp forms on $\GSp_4(\A)$ that sends $v_0$ to $\phi$. 

Consider the functional $\ell: V_{\pi_p} \rightarrow \C$ given by $\ell(v) = H(\pi(g) \iota_p(v))$.  Since the elements in $\GSp_4(\A_\Omega)$ and $\GSp_4(\Q_p)$ commute, $\ell$ is a $(U,\chi)$-functional on $V_{\pi_p}$.  Furthermore, $\phi_\chi(g_p g) = \ell(\pi(g_p) v_0)$.  Hence by Theorem \ref{KS:unram}, the inner integral in (\ref{BLarg}) is equal to $L(\pi_p,\Spin,2s -\frac{1}{2}) \ell(v_0)$.  But $\ell(v_0) = \phi_\chi(g)$, so $I_{\Omega \cup \{p\}}(\phi,s) = I_{\Omega}(\phi,s) L(\pi_p,\Spin, 2s-\frac{1}{2})$.  Taking the limit over $\Omega \supseteq S$ yields
\[I(\phi,s)= I_S(\phi,s) L^S(\pi,\Spin, 2s -\frac{1}{2})\]
as desired.\end{proof}

For a finite prime $p$, denote by $\mathcal{H}_p$ the Hecke algebra of locally constant, compactly supported functions on $\GSp_4(\Q_p)$.  The integration over the bad finite places in $I_S$ may be controlled using the following proposition.
\begin{proposition}\label{KS:ram} For any $v_p$ in $V_{\pi_p}$, there exists a Schwartz function $\Phi_p$ on $W_4(\Q_p)$, a Schwartz function $\alpha_p$ on $V_5(\mathbf{Q}_p)$, and an element $\beta_p$ in $\mathcal{H}_p$ such that for all $(U,\chi)$-functionals $\ell$,
\[\int_{U(\mathbf{Q}_p)Z(\mathbf{Q}_p)\backslash \GSp_4(\mathbf{Q}_p)}{f_p^{\Phi_p}(g,s)\alpha_{\chi,p}(g) \ell(\pi(g)(\beta_p * v_p)) \,dg} = \ell(v_p).\]
\end{proposition}

The proposition is proved in Section \ref{ramified}. One also has the analogous statement for the integral $J$.  As in Corollary \ref{partialL}, one can use this proposition to control $I_S(\phi,s)$.

\begin{corollary}\label{badCor} Suppose the cusp form $\phi$ and the finite set of places $S$ are as in Corollary \ref{partialL}.  Then there exists $\beta = \otimes_p{\beta_p}$ in $\otimes_{p \in S\setminus \{\infty\}}\mathcal{H}_p$ together with $\Phi_p$ and $\alpha_p$ for all $p \in S \setminus \{\infty\}$ such that $I_S(\beta*\phi,s) = I_\infty(\phi,s)$.\end{corollary}

\begin{proof} Since $\phi$ corresponds to a factorizable element under some choice of isomorphism $\pi \cong \otimes_p\pi_p$, for each $p$ in $S \setminus\{\infty\}$ there exists a $\pi_p$-equivariant map from $V_{\pi_p}$ into the space of cusp forms on $\GSp_4(\A)$ and an element $v_p$ in $V_{\pi_p}$ such that $\iota_p(v_p) = \phi$.  Let $\beta_p$, $\Phi_p$, and $\alpha_p$ be functions that satisfy the conclusion of Proposition \ref{KS:ram} for this $v_p$.

Suppose that $\Omega$ is a finite set of places contained in $S$ and containing $\infty$, and set $\beta_\Omega = \prod_{v \in \Omega \setminus \{\infty\}}{\beta_v}$.  We will show $I_{\Omega \cup \{p\}}(\beta_{\Omega \cup \{p\}} *\phi,s) = I_{\Omega}(\beta_{\Omega} *\phi,s)$.  Granted this equality, the corollary follows by descending induction on $\Omega$.

As in the proof of Corollary \ref{partialL}, one has
\begin{align}\label{BLarg:ram} I_{\Omega \cup \{p\}}(\beta_{\Omega \cup \{p\}} *\phi,s) &= \int_{U(\mathbf{A}_\Omega)Z(\mathbf{A}_\Omega)\backslash \GSp_4(\mathbf{A}_\Omega)}{f_\Omega^{\Phi}(g,s)\alpha_{\chi,\Omega}(g)} \nonumber \\ & \times \left( \int_{U(\Q_p)Z(\Q_p)\backslash \GSp_4(\Q_p)}{f_p^{\Phi_p}(g_p,s)\alpha_{\chi,p}(g_p)(\beta_{\Omega \cup \{p\}} *\phi)_\chi(g_p g)\,dg_p}\right)\,dg.\end{align}

Fix $g \in \GSp_4(\A_\Omega)$.  Define a functional $\ell: V_{\pi_p} \rightarrow \C$ by $\ell(v) = H(\pi(g) \beta_\Omega * \iota_p(v))$.  Then $\ell$ is a $(U,\chi)$-functional on $V_{\pi_p}$ and $(\beta_{\Omega \cup \{p\}} *\phi)_\chi(g_p g) = \ell( \pi_p(g_p) \beta_p * v_p)$.  Since $\beta_p$, $\Phi_p$, and $\alpha_p$ satisfy the conclusion of Proposition \ref{KS:ram}, the inner integral in (\ref{BLarg:ram}) is $\ell(v_p) = (\beta_\Omega * \phi)_\chi(g)$.  Hence $I_{\Omega \cup \{p\}}(\beta_{\Omega \cup \{p\}} *\phi,s) = I_{\Omega}(\beta_{\Omega} *\phi,s)$, completing the proof.\end{proof}

It follows immediately from Corollary \ref{partialL} and Corollary \ref{badCor} that the data $(\Phi,\phi', \alpha)$ for the integral may be chosen so that
\[I(\phi',s) = I_\infty(\phi,s)L^S(\pi,\Spin,2s-\frac{1}{2}).\]

We also verify in Section \ref{ramified} that the integral $J(\phi,s)$ can be made non-vanishing for any cuspidal automorphic representation.  Namely, we prove the following proposition.
\begin{proposition}\label{KS:arch}
Suppose that $\phi = v_\infty \otimes v_f$ is factorizable under some choice of factorization $\pi \cong \pi_\infty \otimes \pi_f$ and that $H(\phi) \neq 0$.  Then there is a smooth vector $v_\infty'$ in the space of $\pi_\infty$ and a Schwartz function $\Phi_\infty$ in $\mathcal{S}(W_4(\R))$ such that if $\phi' = v_\infty' \otimes v_f$, then the Archimedean integral
\begin{equation}\label{eqn:archint} J_\infty(\phi',s) = \int_{N(\mathbf{R})Z(\mathbf{R})\backslash\GL_{2,L}^*(\mathbf{R})} H(\pi(g)\phi')f_\infty^{\Phi_\infty}(g,s)dg \end{equation}
is not identically zero.
\end{proposition}
One can prove the analogous result for the integral $I(\phi,s)$.
\subsection{Main theorems}
We summarize the findings above in the first main theorem.
\begin{theorem} \label{thm:mainthm1} Suppose $\pi$ is an automorphic cuspidal representation of $\PGSp_4(\mathbf{A})$, and $\phi$ is a cusp form in the space of $\pi$.  Then the data $\alpha_f$, $\Phi$, and $\phi' \in V_\pi$ may be chosen so that 
\[I(\phi',s) = I_{\infty}(\phi,s)L^S(\pi,\Spin,2s-\frac{1}{2})\]
for a sufficiently large set of finite primes $S$.  Similarly, the data may be chosen so that 
\[J(\phi',s) = J_{\infty}(\phi,s)L^S(\pi,\Spin,2s-\frac{1}{2})\]
for a sufficiently large set of finite primes $S$.
\end{theorem}
In Theorem \ref{thm:mainthm1}, the Archimedean integral $I_{\infty}(\phi,s)$ is given by (\ref{IOmega}) with $\Omega = \{\infty\}$, and similarly for $J_{\infty}(\phi,s)$.  Observe that these integrals involve the actual cusp form $\phi$ rather than an arbitrary $(U,\chi)$-functional $\ell$.

If $\phi$ comes from a level one holomorphic Siegel modular form, then $\pi_p$ is spherical for all finite primes $p$.  (See Section \ref{hol} for a precise statement.)  Thus we can select $\Phi_p$ and $\alpha_p$ for all finite primes $p$ so that the set of bad places $S$ in Theorem \ref{thm:mainthm1} is empty.  Furthermore, we can even choose $\alpha_\infty$ and $\Phi_\infty$ in such a way that we can calculate $I_{\infty}(\phi,s)$ explicitly in terms of $\Gamma$-functions.
\begin{theorem} \label{archimedean}Suppose $\phi$ comes from a level one holomorphic Siegel modular form $f_\phi$ of weight\footnote{In the presence of our level one assumption, the condition $r \geq 6$ is redundant, since the smallest weight of a level one holomorphic Siegel modular cusp form on $\GSp_4$ is $10$.} $r \geq 6$, as made precise in Section \ref{hol}.  Set $T =   \mm{-D}{}{}{1}$ if $D \not \equiv 1 \pmod{4}$ and $T = \mm{\frac{1-D}{4}}{\frac{1}{2}}{\frac{1}{2}}{1}$ if $D \equiv 1 \pmod{4}$.  Denote by $a(T)$ the Fourier coefficient of $f_\phi$ corresponding to the two-by-two symmetric matrix $T$.  Then $\alpha_\infty, \Phi_\infty$ may be chosen so that
\[I_{\infty}(\phi,s) = a(T)\pi^{-2s}(4\pi)^{-(2s+r-2)}\Gamma(2s)\Gamma(2s+r-2).\]
Consequently, with this choice of data, 
\[I(\phi,s) = a(T)\pi^{-2s}(4\pi)^{-(2s+r-2)}\Gamma(2s)\Gamma(2s+r-2)L(\pi,\Spin,2s-\frac{1}{2}).\]
\end{theorem}
This theorem will be proved in Section \ref{hol}.
%=================================================
\section{Unramified calculation}\label{unramified}
%=================================================
In this section, we carry out a purely local non-Archimedean calculation to prove Theorem \ref{KS:unram}.  Thus $p$ in this section always denotes a fixed finite place.  Both here as well as in Section \ref{ramified}, we will find it convenient to have a notation for the characteristic functions of sets.  If $\mathcal{P}$ is a condition on, say, the set of four-by-four matrices, then $\charf(\mathcal{P})$ denotes the characteristic function of the set where $\mathcal{P}$ is satisfied.  For instance, $\charf(\det(g) \neq 0)$ denotes the characteristic function of $\GL_4$ inside $M_4$.  Since we will use $\charf(g \in M_4(\mathbf{Z}_p))$ very frequently, we write just $\charf(g)$ for this characteristic function of $M_4(\Z_p)$.  The exact domain of definition of the functions $\charf(\mathcal{P})$ will always be clear from the context. 

The general strategy for proving results such as Theorem \ref{KS:unram} goes back to Piatetski-Shapiro and Rallis \cite{psr}.  Define
\begin{equation}\label{delta_Defn}\Delta^s(g) = |\nu(g)|^s \charf(g), \end{equation}
where $s \in \mathbf{C}$ and $g\in \GSp_4(\mathbf{Q}_p)$.  One starts with the explicit determination of $L(\pi_p,\Spin,s)$ in terms of Hecke operators, which is codified in the following statement that relates this local $L$-function to $\Delta^s$.
\begin{proposition}\label{omega_p} Suppose $\Omega_p(g)$ is the Macdonald spherical function for the unramified representation $\pi_p$, normalized so that $\Omega_p(1) = 1$.  Then
\[\int_{\GSp_4(\mathbf{Q}_p)}{\Delta^s(g) \Omega_p(g) \,dg} = (1 - \omega_\pi(p)p^{2-2s})L(\pi_p,\Spin,s-\frac{3}{2}),\]
where $\omega_\pi$ is the central character of $\pi_p$.
\end{proposition}
This is a classical result of Shimura \cite[Theorem 2]{shimura} in the theory of the rank two symplectic group.  A proof can be found in \cite[Proposition 3.53]{a2}.  Since the statement and proof in \cite{a2} is in a semi-classical framework, we give a self-contained argument in Appendix \ref{localSpin} as well as definitions of $\Omega_p$ and the local Spin $L$-function.

Suppose now that $\ell$ is a $(U,\chi)$-functional as defined in Section \ref{subsec:nonuniquemodels}.  If $v_0$ is a spherical vector for $\pi_p$, it follows from Proposition \ref{omega_p} that
\begin{align*}\int_{\GSp_4(\mathbf{Q}_p)}{\Delta^s(g) \ell(\pi(g) v_0) \,dg} &= \int_{\GSp_4(\mathbf{Z}_p) \backslash \GSp_4(\mathbf{Q}_p)}{\Delta^s(g) \int_{\GSp_4(\mathbf{Z}_p)} \ell(\pi(kg) v_0)\,dk \,dg}\\ &= \ell(v_0)\int_{\GSp_4(\mathbf{Q}_p)}{\Delta^s(g) \Omega_p(g) \,dg} \\ &= \ell(v_0)(1 - \omega_\pi(p)p^{2-2s})L(\pi_p,\Spin,s-\frac{3}{2}),\end{align*}
where we have used the invariance of $\Delta$ under left $\GSp_4(\mathbf{Z}_p)$-multiplication in the first equality and, in the second, the fact that 
\begin{equation}\label{ell_Omega}\ell(v_0)\Omega_p(g) = \int_{\GSp_4(\mathbf{Z}_p)}{\ell(\pi(kg) v_0)\,dk}. \end{equation}
To see the validity of (\ref{ell_Omega}), note that the right-hand side of this equation defines a bi-$K$-invariant function on $\GSp_4(\Q_p)$, whose value at $g = 1$ is $\ell(v_0)$.  Since the eigenvalues of this function under the action of the spherical Hecke algebra of $\GSp_4(\Q_p)$ agree with those of a nonzero spherical vector in the space of $\pi_p$, the equality follows.  (When $\ell(v_0) = 0$, this argument shows that both sides of (\ref{ell_Omega}) are $0$.)

Thus, to prove $I_p(\ell,s) = \ell(v_0)L(\pi_p,\Spin,2s-\frac{1}{2})$, it suffices to prove 
\begin{equation}\label{eqn:mainunrameq} (1 - \omega_\pi(p)p^{2-2s})I_p\left(\ell,\frac{s-1}{2}\right) = \int_{\GSp_4(\mathbf{Q}_p)}{\Delta^s(g) \ell(\pi(g) v_0) \,dg}.\end{equation}
We will prove this by explicitly calculating both sides.

Recall the definition of $\alpha_{\chi,p}$ from (\ref{eqn:alphachivdef}).  We now calculate $\alpha_{\chi,p}$ when $\alpha_p$ is the characteristic function of $V_5(\Z_p)$. To succinctly express the result, we first define functions $\Delta_0$ and $\Delta_0'$ as follows.  First, by the Iwasawa decomposition, any $m \in \GL_2(\mathbf{Q}_p)$ can be transformed by right $\GL_2(\mathbf{Z}_p)$-multiplication to a matrix of the form $\mm{m_{11}}{m_{12}}{}{m_{22}}$.  We then set
\begin{equation} \label{delta_0prime} \Delta_0'(m)=\begin{cases} 1 & \textrm{if }m_{22} | m_{11}, m_{22} | m_{12},\textrm{ and }(\frac{m_{12}}{m_{22}})^2 \equiv D \pmod{\frac{m_{11}}{m_{22}}}\\ 0 & \textrm{otherwise}\end{cases}\end{equation}
if $D \not\equiv 1 \pmod{4}$ and
\begin{equation} \label{delta_0prime1mod4} \Delta_0'(m)=\begin{cases} 1 & \textrm{if }m_{22} | m_{11}, m_{22} | 2m_{12},\textrm{ and }(\frac{m_{12}}{m_{22}})^2-\frac{m_{12}}{m_{22}} \equiv \frac{D-1}{4} \pmod{\frac{m_{11}}{m_{22}}}\\ 0 & \textrm{otherwise}\end{cases}\end{equation}
if $D \equiv 1 \pmod{4}$.  This definition is independent of the choice of $m_{11},m_{12}$, and $m_{22}$.  For a matrix
\[m = \mb{m_t}{}{}{m_b}\]
in the Levi of the Siegel parabolic of $\GSp_4$, so that $m_t, m_b$ are $2\times 2$ matrices, we define 
\begin{equation} \label{delta_0} \Delta_0(m) = \charf\left(\left|\frac{\det(m_t)}{\det(m_b)}\right| \leq 1\right)\left|\frac{\det(m_t)}{\det(m_b)}\right|^{1/2}\Delta_0'(m_b).\end{equation}

\begin{proposition}\label{KS:alphachi} Suppose that $\alpha_p$ is the characteristic function of $V_5(\mathbf Z_p)$.  Then if $m$ is in the Levi of the Siegel parabolic, $\alpha_{\chi,p}(m) = \Delta_0(m)$.
\end{proposition}

\begin{proof} Suppose
\begin{equation} \label{eqn:mn24eq} m =\left(\begin{array}{cc|cc} \lambda& & &\\ &\lambda& & \\ \hline & &1& \\ & & &1\end{array}\right)\left(\begin{array}{cc|cc} a&b & &\\ c&d& & \\ \hline & &d&-c \\ & & -b&a\end{array}\right), \quad n_{24}(z) = \left(\begin{array}{cc|cc} 1& & &\\ &1& &z\\ \hline & &1& \\ & & &1\end{array}\right).\end{equation}
Then if $D \not\equiv 1 \pmod{4}$, we have
\begin{align*}
	v_Dn_{24}(z)m =& \delta^{-1}(cd-Dab)(e_1\wedge f_1 - e_2 \wedge f_2) + \delta^{-1}(Da^2-c^2) e_1 \wedge f_2 + \delta^{-1}(d^2-Db^2)e_2 \wedge f_1\\
	&- \frac{z}{\lambda} f_1 \wedge f_2,
\end{align*}
where $\delta = ad-bc$.  Since $\alpha_{\chi,p}(m)$ is invariant under right $\GSp_4(\mathbf{Z}_p)$-multiplication, we may assume $b = 0$.  Then
\[v_Dn_{24}(z)m = \frac{c}{a}(e_1\wedge f_1 - e_2 \wedge f_2) + \frac{Da^2-c^2}{ad} e_1 \wedge f_2 + \frac{d}{a} e_2 \wedge f_1- \frac{z}{\lambda} f_1 \wedge f_2.\]
Thus, under our assumption on $\alpha_p$, for the integral
\[\int_{\mathbf{Q}_p}{\psi(z)\alpha_p(v_Dn_{24}(z)m)\,dz} = \alpha_{\chi,p}(m)\]
to be nonzero, it is necessary to have $\frac{c}{a}, \frac{Da^2-c^2}{ad}$, and $\frac{d}{a}$ in $\mathbf{Z}_p$.  If these hold, by orthogonality of characters on the compact subgroup of $\mathbf{Q}_p$ where $\alpha_p(v_Dn_{24}(z)m)$ is supported, the integral is nonzero exactly when $|\lambda|\le 1$ and in this case equal to $|\lambda|$.  The conditions for the first three coefficients are exactly the conditions defining $\Delta_0'(m_b)$.  Since $|\lambda| = \left|\frac{\det(m_t)}{\det(m_b)}\right|^{1/2}$, the proposition follows.

If $D \equiv 1 \pmod{4}$, we have
\begin{align*}
	v_Dn_{24}(z)m =& \(\delta^{-1}cd-\frac{D-1}{4}\delta^{-1}ab+\frac{1}{2}\)(e_1\wedge f_1 - e_2 \wedge f_2) + \delta^{-1}\(\frac{D-1}{4}a^2-ac-c^2\) e_1 \wedge f_2\\
	&+ \delta^{-1}\(d^2+db-\frac{D-1}{4}b^2\)e_2 \wedge f_1- \frac{z}{\lambda} f_1 \wedge f_2.
\end{align*}
When $b=0$, this is
\[\(\frac{2c}{a}+1\)\frac{1}{2}(e_1\wedge f_1 - e_2 \wedge f_2)+\frac{\frac{D-1}{4}-\frac{c}{a}-(\frac{c}{a})^2}{\frac{d}{a}} e_1 \wedge f_2 + \frac{d}{a}e_2 \wedge f_1- \frac{z}{\lambda} f_1 \wedge f_2.\]
Reasoning as above and using the different integral structure on $V_5(\mathbf{Z}_p)$ in this case yields the proposition.
\end{proof}

Recall the formula (\ref{f_pdef}) for the section $f_p^\Phi(g,s)$ appearing in the definition of the Eisenstein series.  Applying (\ref{f_pdef}), we now compute
\begin{align*} I_p(\ell,s) &= \int_{U(\mathbf{Q}_p)Z(\mathbf{Q}_p)\backslash \GSp_4(\mathbf{Q}_p)}{f_p^\Phi(g,s,\Phi)\alpha_{\chi}(g) \ell(\pi(g)v_0) \,dg} \\
	&= \int_{M(\mathbf{Q}_p)}{\delta_P^{-1}(g)|\nu(g)|^{2s}\Phi(f_2 g)\alpha_\chi(g)\ell(\pi(g) v_0) \,dg} \\
	&= \int_{M(\mathbf{Q}_p)}{\delta_P^{-1}(g)|\nu(g)|^{2s}\Delta_0(g)\charf(g)\ell(\pi(g) v_0) \,dg}.
\end{align*}
For the last equality, it follows easily from the definition in (\ref{delta_0}) that if $\Delta_0(g) \neq 0$, then $\Phi(f_2g) \neq 0$ if and only if $\charf(g) \neq 0$. For an element $g$ of $\GL_4(\mathbf{Q}_p)$, we define $\val_p(g)$ to be the unique integer such that $g = p^{\val_p(g)}g_0$ with $g_0$ in $M_4(\mathbf{Z}_p) \setminus pM_4(\mathbf{Z}_p)$.  Then
\begin{align} I_p(\ell,s)&= \sum_{k \geq 0}{(|p|^{4s}\omega_\pi(p))^k}\int_{M(\mathbf{Q}_p)}{\delta_P^{-1}(g)|\nu(g)|^{2s}\Delta_0(g)\charf(\val_p(g) = 0)\ell(\pi(g) v_0) \,dg} \nonumber \\
	&= \left(1-\omega_\pi(p)p^{-4s}\right)^{-1}\int_{M(\mathbf{Q}_p)}{\delta_P^{-1}(g)|\nu(g)|^{2s}\Delta_0(g)\charf(\val_p(g) = 0)\ell(\pi(g) v_0) \,dg}.\label{eqn:ipsum} \end{align}
Shifting the $s$ parameter, we deduce that
\begin{equation}\label{eqn:unramcalclhsI}(1-\omega_\pi(p)p^{2-2s})I_p\(\ell,\frac{s-1}{2}\)=\int_{M(\mathbf{Q}_p)}{\delta_P^{-1}(g)|\nu(g)|^{s-1}\Delta_0(g)\charf(\val_p(g) = 0)\ell(\pi(g) v_0) \,dg}.\end{equation}

We now consider the right-hand side of (\ref{eqn:mainunrameq}), which is $\int_{\GSp_4(\mathbf{Q}_p)}{\Delta^s(g) \ell(\pi(g) v_0) \,dg}$.  By the Iwasawa decomposition, we may rewrite this integral as
\[\int_{M(\mathbf{Q}_p)}\delta_P^{-1}(g)|\nu(g)|^s\ell(\pi(g)v_0)\left\{\int_{U(\mathbf{Q}_p)}\chi(u)\charf(ug) \,du\right\}\,dg.\]
We study the conditions under which the inner integral can be nonvanishing.
\begin{lemma} \label{lem:nonvanishingconds} Define $K_M = \GSp_4(\mathbf{Z}_p) \cap M(\mathbf{Q}_p)$, let $g$ be in $M(\mathbf{Q}_p)$, and suppose that $g$ is equivalent by right $K_M$-multiplication to a matrix with bottom right $2 \times 2$ block of the form $\mm{p^\alpha}{b}{}{p^\delta}$.  Then
\begin{equation}\label{eqn:innerint}\int_{U(\mathbf{Q}_p)}{\chi(u)\charf(ug) \,du}\end{equation}
is $0$ unless $\delta = 0$, $\alpha \ge 0$, and $b^2 \equiv D \pmod{p^\alpha}$ or $b^2 - b \equiv \frac{D-1}{4} \pmod{p^\alpha}$ in the respective cases that $D \not\equiv 1 \pmod{4}$ or $D \equiv 1 \pmod{4}$.  When these conditions hold, the integral is $p^\alpha$.
\end{lemma}
\begin{proof} The set $U_g$ defined by $U_g=\{u \in U(\mathbf Q_p) |\charf(ug) \neq 0\}$ forms a compact abelian group.  Hence by orthogonality of characters, the integral vanishes unless $\chi$ is identically 1 on $U_g$.  If $\delta > 0$ then the matrix with upper right block $\mm{0}{0}{0}{p^{-1}}$, whose image is nontrivial under $\chi$, is in $U_g$, so the integral vanishes.  So we must have $\delta = 0$, and $\alpha \ge 0$ is forced by the integrality of $ug$.  The matrix $u \in U(\mathbf{Q}_p)$ with upper right block $\mm{p^{-\alpha}}{-bp^{-\alpha}}{-bp^{-\alpha}}{b^2p^{-\alpha}}$ is in $U_g$.  As $\chi(u)$ is $\psi(\frac{b^2-D}{p^\alpha})$ or $\psi(\frac{b^2-b+ \frac{1-D}{4}}{p^\alpha})$ if $D \not\equiv 1 \pmod{4}$ or $D \equiv 1 \pmod{4}$ respectively, we obtain the third condition.  Finally, when all the conditions are satisfied, suppose $u$ has upper right block $\mm{u_{11}}{u_{12}}{u_{12}}{u_{22}}$.  Then one sees that for $u$ to be in $U_g$ we must have $u_{11} \in p^{-\alpha}\mathbf{Z}_p$, and then that $u_{12}$, $u_{22}$ are determined modulo $\mathbf{Z}_p$.  The third condition forces $\chi$ to be trivial on $U_g$, so the integral is the measure of $U_g$, which is $p^\alpha$.
\end{proof}
Although Lemma \ref{lem:nonvanishingconds} treats $g$ of a special form, the integral (\ref{eqn:innerint}) is invariant by right $K_M$-multiplication, so there is in fact no loss of generality.  Note also that when $p$ is inert in $L$, the only $2 \times 2$ matrix satisfying the conditions given in Lemma \ref{lem:nonvanishingconds} is the identity.  

We next need to understand the relationship of the model with the various integrality conditions that appeared in (\ref{eqn:unramcalclhsI}).

\begin{lemma} \label{lem:lambdaint}
Suppose that $g =\mm{m_t}{}{}{m_b} \in M(\mathbf{Q}_p)$, where $m_t,m_b \in M_2(\mathbf{Q}_p)$.  If $\ell(\pi(g)v_0)\neq 0$, then $|\frac{\det(m_t)}{\det(m_b)}| \leq 1$.
\end{lemma}

\begin{proof}
For $u \in U(\mathbf{Z}_p)$, we have
\[ \ell(\pi(g) v_0) = \ell(\pi(g) u v_0) = \chi(gug^{-1}) \ell(\pi(g) v_0).\]
Thus if $\ell(\pi(g) v_0) \neq 0$, we must have $\chi(gug^{-1}) = 1$ for all $u \in U(\mathbf{Z}_p)$.  Define $X = \mm{-D}{}{}{1}$ or $\mm{\frac{1-D}{4}}{\frac{1}{2}}{\frac{1}{2}}{1}$ if $D \not\equiv 1 \pmod{4}$ or $D \equiv 1 \pmod{4}$, respectively.  Then if $u = \mm{\mathbf{1}_2}{n_u}{}{\mathbf{1}_2}$ for $n_u \in M_2(\mathbf{Z}_p)$, we have $\chi(gug^{-1}) = \psi(\tr(m_tn_um_b^{-1}X)) = \psi(\tr(n_um_b^{-1}Xm_t))$. Hence $\ell(\pi(g) v_0) \neq 0$ implies $m_b^{-1}Xm_t$ is half-integral.  We will use this to show that $|\frac{\det(m_t)}{\det(m_b)}| \leq 1$.

Assume that $g$ is written in the form specified by (\ref{eqn:mn24eq}).  We may assume without loss of generality that $m_b$ is upper-triangular, i.e. $b=0$, since the hypotheses and results are insensitive to right-multiplication by an element of $M(\mathbf{Z}_p)$.  Then we define $Y$ to be the matrix
\begin{equation}\label{Yeqn}Y=\lambda^{-1} m_b^{-1}Xm_t = \mm{\frac{(\frac{c}{a})^2-D}{\frac{d}{a}}}{\frac{c}{a}}{\frac{c}{a}}{\frac{d}{a}} \textrm{ or }\mm{\frac{\frac{1-D}{4}+\frac{c}{a}+(\frac{c}{a})^2}{\frac{d}{a}}}{\frac{1}{2}+\frac{c}{a}}{\frac{1}{2}+\frac{c}{a}}{\frac{d}{a}}\end{equation}
in the respective cases $D \not\equiv 1 \pmod{4}$ or $D\equiv 1 \pmod{4}$.  We write $x = \frac{d}{a}$ and $y = \frac{c}{a}$ so that
\begin{equation}\label{Ydef}Y = \mm{\frac{y^2-D}{x}}{y}{y}{x} \textrm{ or } \mm{\frac{y^2+y+\frac{1-D}{4}}{x}}{\frac{1}{2}+y}{\frac{1}{2}+y}{x}.\end{equation}

We denote the set of half-integral $2 \times 2$ symmetric matrices by $\Sym_2(\Z_p)^\vee$.  Above we showed that $m_b^{-1}Xm_t=\lambda Y \in \Sym_2(\Z_p)^\vee$.  Since $|\frac{\det(m_t)}{\det(m_b)}| = |\lambda|^2$, we need $|\lambda| \leq 1$.  Thus it suffices to check that $Y \notin p \Sym_2(\Z_p)^\vee$.

First assume $D \not\equiv 1 \pmod{4}$.  If $\frac{y^2-D}{x}$ and $x\in p\mathbf{Z}_p$, then $y^2-D \equiv 0 \pmod{p^2}$.  It follows that if $p\nmid 2D$, $y$ is a unit. Hence $Y \notin p\Sym_2(\Z_p)^\vee$.  If $p|2D$, the congruence $y^2 -D \equiv 0 \pmod{p^2}$ has no solution since $D$ is square-free and $D \not\equiv 1 \pmod{4}$.  We again obtain $Y \notin p\Sym_2(\Z_p)^\vee$.

Now assume $D \equiv 1 \pmod{4}$.  As above, if $\frac{y^2+y+\frac{1-D}{4}}{x},x \in p\mathbf{Z}_p$, we have $y^2+y \equiv \frac{D-1}{4} \pmod{p^2}$.  In particular, $y$ is $p$-integral.  Thus if $p = 2$, the entry $\frac{1}{2} +y$ of $Y$ shows that $Y \notin p\Sym_2(\Z_p)^\vee$.  If $p >2$, we rearrange to obtain $(y + \frac{1}{2})^2 \equiv D \pmod{p^2}$.  Just as in the $D \not\equiv 1 \pmod{4}$ case, if $p |D$ then this congruence has no solution since $D$ is square free, and if $p\nmid D$ then $y + \frac{1}{2}$ is a $p$-adic unit.  So $Y \notin p\Sym_2(\Z_p)^\vee$ in all cases, finishing the proof.
\end{proof}

We can now prove the following.
\begin{lemma} \label{lem:finalIcalc} If $g \in M(\mathbf{Q}_p)$ and $\ell(\pi(g) v_0) \neq 0$, then
\[\int_{U(\mathbf{Q}_p)}{\chi(u)\charf(ug) \,du} = |\nu(g)|^{-1}\Delta_0(g) \charf(\val_p(g) = 0)\]
\end{lemma}
\begin{proof}  Write $g = \mm{m_t}{}{}{m_b}$, and assume without loss of generality that $m_b = \mm{p^\alpha}{b}{}{p^\delta}$.  Then when both sides are non-zero, Lemma \ref{lem:nonvanishingconds} shows that they are equal.  It remains to check that the vanishing conditions on the two sides are the same.  If the right-hand side is non-zero, $m_b \in M_2(\mathbf{Z}_p)$.  Since $\ell(\pi(g) v_0)\neq 0$, Lemma \ref{lem:lambdaint} verifies that $|\frac{\det(m_t)}{\det(m_b)}| \leq 1$.  It follows from this inequality and the condition $\val_p(g) = 0$ that some entry of $m_b$ is a $p$-adic unit.  The conditions defining $\Delta_0(g)$ then force $\delta=0$ and either the congruence $b^2 \equiv D \pmod{p^\alpha}$ or $b^2+b \equiv \frac{D-1}{4} \pmod{p^\alpha}$.  Conversely, these conditions verify that $m_b \in M_2(\mathbf{Z}_p)$ and provide a unital entry in $m_b$.  Using Lemma \ref{lem:lambdaint}, we have $|\frac{\det(m_t)}{\det(m_b)}| \leq 1$, so $m_t$ is integral.  It follows that $\val_p(g) = 0$ and $\Delta_0(g)\ne 0$. \end{proof}

We may now compute
\begin{align*} \int_{\GSp_4(\mathbf{Q}_p)}{\Delta^s(g) \ell(\pi(g) v_0) \,dg}&= \int_{M(\mathbf{Q}_p)}\delta_P^{-1}(g)|\nu(g)|^s\ell(\pi(g)v_0)\left\{\int_{U(\mathbf{Q}_p)}\chi(u)\charf(ug) \,du\right\}\,dg \\
&= \int_{M(\mathbf{Q}_p)}\delta_P^{-1}(g)|\nu(g)|^s|\nu(g)|^{-1}\Delta_0(g)\charf(\val_p(g) = 0)\ell(\pi(g)v_0)\,dg.\end{align*}
This last line is what we obtained in (\ref{eqn:unramcalclhsI}) for $(1-\omega_\pi(p)p^{2-2s})I_p(\ell,\frac{s-1}{2})$.  Thus
\[\int_{\GSp_4(\mathbf{Q}_p)}{\Delta^s(g) \ell(\pi(g) v_0) \,dg} = (1-\omega_\pi(p)p^{2-2s})I_p\left(\ell,\frac{s-1}{2}\right),\]
completing the unramified calculation for the integral $I$.

We now treat the integral $J$ by reducing the unramified calculation to the one completed above.  We have already computed the right-hand side of (\ref{eqn:mainunrameq}).  We thus need to compute
\[J_p\(\ell, \frac{s-1}{2}\) = \int_{N(\mathbf{Q}_p)Z(\mathbf{Q}_p)\backslash \GL_{2,L}^*(\mathbf{Q}_p)} f_p^\Phi\(g,\frac{s-1}{2}\) \ell(\pi(g)v_0)\,dg.\]
Substituting the definition (\ref{f_pdef}), applying the Iwasawa decomposition, and rearranging, we obtain
\[J_p\(\ell,\frac{s-1}{2}\) = \int_{T_L^*(\mathbf{Q}_p)} \delta_{B_L^*}^{-1}(g)|\nu(g)|^{s-1}\Phi(f_2g)\ell(\pi(g)v_0)\,dg.\]
Here $T_L^*$ and $B_L^*$ denote the torus and Borel of $\GL_{2,L}^*$.  The explicit description of $g$ in (\ref{GGmatrix2}) or (\ref{GGmatrix3}) shows that if $\Phi(f_2g) = 1$, the bottom right block of $g$ is integral (and conversely).  If $\ell(\pi(g)v_0) \ne 0$, Lemma \ref{lem:lambdaint} ensures that $|\frac{\det(m_t)}{\det(m_b)}| \leq 1$ and thus $g \in M_4(\mathbf{Z}_p)$.  We may reduce to the case that $\val_p(g)=0$ by the same process used in (\ref{eqn:ipsum}), which gives
\[J_p\(\ell,\frac{s-1}{2}\) = \left(1-\omega_\pi(p)p^{2-2s}\right)^{-1}\int_{T_L^*(\mathbf{Q}_p)} \delta_P^{-1}(g)|\nu(g)|^{s-1}\charf(\val_p(g)=0)\ell(\pi(g)v_0)\,dg.\]
The unramified calculation now reduces to checking the following proposition.
\begin{proposition}
We have
\begin{align}&\int_{M(\mathbf{Q}_p)}\delta_P^{-1}(g)|\nu(g)|^{s-1}\Delta_0(g)\charf(\val_p(g)=0)\ell(\pi(g) v_0) \,dg \label{eqn:jinteqL}\\ = &\int_{T_L^*(\mathbf{Q}_p)} \delta_{B_L^*}^{-1}(g)|\nu(g)|^{s-1}\charf(\val_p(g)=0)\ell(\pi(g)v_0)\,dg.\label{eqn:jinteq}\end{align}
\end{proposition}

\begin{proof}
We claim that if $h \in \GL_2(\Q_p)$, then the matrix $Y$ defined by $Y=h^{-1}X(\det(h)\,^th^{-1})$ is half-integral if and only if $\Delta_0'(h) \neq 0$, where $X$ is the matrix defined in Lemma \ref{lem:lambdaint}.  To check this, it suffices to assume that $h = \mm{d}{-c}{}{a}$ is upper-triangular.  Then $Y$ was computed in (\ref{Yeqn}), and an inspection of the matrix on the right-hand side of (\ref{Yeqn}) verifies the claim.

We next use the claim to verify that if the integrand in (\ref{eqn:jinteq}) is non-zero, then so is the integrand in (\ref{eqn:jinteqL}).  This amounts to checking that if $\mm{m_t}{}{}{m_b}$ is in $T_L^*(\Q_p)$, then $\Delta_0'(m_b) \neq 0$.  Using the explicit form of the matrix $m_b$ specified in (\ref{GGmatrix2}) or (\ref{GGmatrix3}), one checks immediately that $m_b X (\det(m_b)\,^tm_b^{-1}) = X$ is half-integral, which implies by the claim that $\Delta_0'(m_b) \neq 0$.

We claim that in fact there is a bijection, induced by the embedding $T_L^* \subseteq M$, from the elements of $T_L^*(\mathbf{Q}_p)/T_L^*(\mathbf{Z}_p)$ for which the integrand on the right is nonzero to elements of $M(\mathbf{Q}_p)/M(\mathbf{Z}_p)$ for which the integrand on the left is nonzero.  We have just observed the existence of an injective map and need only check the surjectivity.  Maintain the notation of (\ref{eqn:mn24eq}).  By Lemma \ref{lem:lambdaint}, if $g \in M(\mathbf{Q}_p)$ satisfies $\ell(\pi(g)v_0) \ne 0$ and $\val_p(g)=0$, we have $|\frac{\det(m_t)}{\det(m_b)}| \leq 1$, so some entry of the lower right $2\times 2$ block of $g$ is a unit.  As in the proof of Lemma \ref{lem:finalIcalc}, $\Delta_0(g) \ne 0$ forces $g$ to be equivalent by right $M(\mathbf{Z}_p)$-multiplication to an element of $M_4(\mathbf{Z}_p) \cap M(\mathbf{Q}_p)$ with $m_b = \mm{p^\alpha}{b}{}{1}$, where $b^2 \equiv D \pmod{p^\alpha}$ or $b^2-b\equiv \frac{D-1}{4}\pmod{p^\alpha}$ depending on $D \pmod{4}$ as before.  The coset in $M(\mathbf{Q}_p)/M(\mathbf{Z}_p)$ is determined by $\alpha,b$, and $\lambda$.  (Note that changing $b$ by a multiple of $p^\alpha$ does not change the coset.)  Given choices of $\alpha,b$, and $\lambda$ meeting the aforementioned conditions, we will produce an element of $T_L^*(\mathbf{Q}_p)$ equivalent to $m$ by right $M(\mathbf{Z}_p)$-multiplication.

Assume that $D \not\equiv 1 \pmod{4}$.  We will choose an element $m_b' \in M_2(\mathbf{Z}_p)$ such that if one sets $m_t'=\lambda \det(m_b') {}^t(m_b')^{-1}$, the matrix $m' = \mm{m_t'}{}{}{m_b'}$ is in $T_L^*(\mathbf{Q}_p)\cap mM(\mathbf{Z}_p)$.  If $p|D$, the relation $b^2 \equiv D \pmod{p^\alpha}$ forces $\alpha = 0$ or $\alpha = 1$ and (possibly after changing $m$ within its coset) $b=0$.  Then $m_b' = \mm{D}{}{}{1}$ puts $m'$ in the same coset as $m$.  If $p=2$ and $D$ is odd, the condition $D \not\equiv 1 \pmod{4}$ again forces $\alpha=0$ or $\alpha =1$.  The case $\alpha=0$ is trivial, and if $\alpha = 1$ we may assume $b=1$.  Setting $m_b' = \mm{1}{D}{1}{1}$ again puts $m'$ in the same coset as $m$.  Finally, we assume $p\nmid 2D$.  We define an element $m'(\gamma) \in T_L^*(\mathbf{Q}_p)$, where $\gamma \in \mathbf{Z}_p$, by defining the lower right block to be $m_b'(\gamma)=\mm{b+p^\alpha\gamma}{D}{1}{b+p^\alpha\gamma}$ and the upper left to be $m_t'(\gamma)=\lambda \det(m_b'(\gamma)) {}^t(m_b'(\gamma))^{-1}$.  Then after applying column operations, $m_b$ transforms to $\mm{D-(b+\gamma p^\alpha)^2}{b+\gamma p^\alpha}{}{1}$.  We have
\[D-(b+\gamma p^\alpha)^2 = D-b^2-\gamma^2p^{2\alpha}+2b\gamma p^\alpha = (\gamma'-\gamma^2p^\alpha+2b\gamma)p^\alpha,\]
where $\gamma'p^\alpha = D-b^2$.  Since $p \nmid 2D$, $2b$ is a unit.  Then for any $\gamma'$, one can choose $\gamma$ so that $\gamma'+2b\gamma$ and thus $\gamma'-\gamma^2p^\alpha+2b\gamma$ is a unit.  Multiplying the first column of $\mm{D-(b+\gamma p^\alpha)^2}{b+\gamma p^\alpha}{}{1}$ by the inverse of this unit and then subtracting $\gamma$ times the first column from the second yields $\mm{p^\alpha}{b}{}{1}$ as needed.  

Now assume that $D \equiv 1 \pmod{4}$.  Examining (\ref{GGmatrix3}), the matrix $m_b'$ must now be of the form $\mm{\epsilon}{\frac{D-1}{4}\eta}{\eta}{\epsilon-\eta}$.  Consider the matrix $m_b'(\gamma) \in M_2(\mathbf{Z}_p)$ defined for $\gamma \in \mathbf{Z}_p$ by $m_b'(\gamma)=\mm{b+\gamma p^\alpha}{\frac{D-1}{4}}{1}{b+\gamma p^\alpha-1}$.  By applying column operations one can transform this into $\mm{(b+\gamma p^\alpha)^2-(b+\gamma p^\alpha)+\frac{1-D}{4}}{b}{}{1}$.  We have
\[(b+\gamma p^\alpha)^2-(b+\gamma p^\alpha)+\frac{1-D}{4} = b^2-b+\frac{1-D}{4}+(2b-1)\gamma p^\alpha+\gamma^2p^{2\alpha} =(\gamma'+(2b-1)\gamma+\gamma^2p^\alpha)p^\alpha,\]
where $\gamma'p^\alpha = b^2-b+\frac{1-D}{4}$.  If $2b-1$ is a unit, then $\gamma$ may be chosen so that $\gamma'+(2b-1)\gamma+\gamma^2p^\alpha$ is a unit and the argument proceeds as above.  If $p|(2b-1)$ (which means $p\ne 2$), the transformed relation $4b^2-4b+1 = (2b-1)^2 \equiv D \pmod{p^\alpha}$ forces either $\alpha=0$ or $p|D$ and $\alpha=1$.  The former case is trivial, and in the latter case we claim that we may use $\mm{b}{\frac{D-1}{4}}{1}{b-1}$.  Indeed, if we subtract $b-1$ times the left column from the right, switch columns, and multiply by the left column by $-1$, we obtain $\mm{b^2-b+\frac{1-D}{4}}{b}{}{1}$.  We have $\ord_p(b^2-b+\frac{1-D}{4})=\ord_p(4b^2-4b+1-D)=1$ since $p^2|(2b-1)^2$ and $p$ exactly divides $D$.

It remains to check that for each $m \in T_L^*(\mathbf{Q}_p)$ for which the right-hand side integrand is non-zero, the two integrands are identical when evaluated at $m$.  In other words, we need $\delta_P^{-1}(m)\left|\frac{\det(m_t)}{\det(m_b)}\right|^{1/2} = \delta_{B_L^*}^{-1}(m)$, which is clear by direct calculation.

\end{proof}

%=================================================
\section{Ramified calculation}\label{ramified}
%=================================================
In this section we will prove Propositions \ref{KS:ram} and \ref{KS:arch}.  Let $p$ be a place of $\Q$, either finite or infinite.  Suppose that $\ell$ is a $(U,\chi)$-functional on $V_{\pi_p}$ and $\Phi_p$ is a Schwartz function on $W_4(\Q_p)$.  (We use the notation $\mathbf{Q}_\infty=\R$.)  For a smooth vector $v_1$ in $V_{\pi_p}$, consider the integral
\begin{equation}\label{eqn:unfoldfinfty}J_{p}(\ell, v_1, s) = \int_{N(\mathbf{Q}_p)Z(\mathbf{Q}_p)\backslash {\GL_{2,L}^*}(\mathbf{Q}_p)}{f_p^{\Phi_p}(g,s)\ell(\pi(g)v_1) \,dg}.\end{equation}

We decompose the domain of the right-hand side integral in (\ref{eqn:unfoldfinfty}) as follows.  Let
\[T' = \set{t \in \GSp_4| t = \diag(y,y,1,1), y \in \mathbf{G}_m}.\]
We have
\begin{align}\label{badfte} J_p(\ell, v_1, s) &= \int_{N(\Q_p) \backslash \GL_{2,L}^*(\Q_p)}{|\nu(g)|^{2s} \Phi_p(f_2 g) \ell(\pi(g) v_1) \,dg} \nonumber \\ &= \int_{T'(\Q_p)}|\nu(t)|^{2s} \int_{N(\Q_p) \backslash \SL_{2,L}(\Q_p)}{\Phi_p(f_2 xt) \ell(\pi(xt) v_1) \,dx}\,dt \\ \label{badfte2} &= \int_{T'(\Q_p)}|\nu(t)|^{2s-2} \int_{N(\Q_p) \backslash \SL_{2,L}(\Q_p)}{\Phi_p(f_2 x) \ell(\pi(tx) v_1) \,dx}\,dt \\ \label{badfte3} &= \int_{N(\Q_p) \backslash \SL_{2,L}(\Q_p)}\Phi_p(f_2 x)\int_{T'(\Q_p)}{|\nu(t)|^{2s-2}  \ell(\pi(tx) v_1) \,dt}\,dx. \end{align}
Here, to go from (\ref{badfte}) to (\ref{badfte2}), we applied the change of variables $x \mapsto txt^{-1}$ and used $f_2 t = f_2$.

Before proving Propositions \ref{KS:ram} and \ref{KS:arch}, we consider one more preparatory fact.  Let $v$ be a smooth vector in $V_{\pi_p}$, let $\eta$ be a Schwartz function on $\Q_p$, and set
\[u(z) = \left(\begin{array}{cc|cc}1& &0&0 \\ & 1&0&z\\ \hline & &1& \\ & & &1\end{array}\right).\]
Suppose
\begin{equation}\label{v1Def} v_1 = \int_{\mathbf G_a(\mathbf{Q}_p)}{\eta(z) \pi(u(z)) v \,dz}. \end{equation}
Then if $t = \diag(y,y,1,1)$,
\begin{align}\label{FTeqn} \ell( \pi(t) v_1 ) &= \left\{\int_{\mathbf G_a(\mathbf{Q}_p)}{\eta(z)\chi(tu(z)t^{-1})\, dz}\right\} \ell(t v) \nonumber \\ &= \left\{\int_{\mathbf G_a(\mathbf{Q}_p)}{\eta(z)\psi(yz)\, dz}\right\} \ell(t v) \\ &= \hat{\eta}(y) \ell(t v). \nonumber \end{align}
Here, $\hat{\eta}$ is the Fourier transform of $\eta$, defined to be the expression in braces in (\ref{FTeqn}).

The following lemma gives the analogue of Proposition \ref{KS:ram} for the integral $J$.  Recall that $\mathcal{H}_p$ denotes the Hecke algebra of locally constant compactly supported functions on $\GSp_4(\Q_p)$.
\begin{proposition}\label{KS:ram1} Suppose $p$ is a finite place.  For any $v$ in $V_{\pi_p}$, there exists a Schwartz function $\Phi_p$ on $W_4(\Q_p)$ and an element $\beta_p$ of $\mathcal{H}_p$ such that for all $(U,\chi)$ functionals $\ell$
\[J_{p}(\ell,\beta_p * v,s) = \int_{N(\mathbf{Q}_p)Z(\mathbf{Q}_p)\backslash {\GL_{2,L}^*}(\mathbf{Q}_p)}{f_p^{\Phi_p}(g,s)\ell(\pi(g)(\beta_p *v)) \,dg} = \ell(v).\]
\end{proposition}
\begin{proof} Denote by $K_N = \{\mathbf{1}_4 + p^NM_4(\mathbf{Z}_p)\} \cap \GSp_4(\mathbf Q_p)$ the congruence subgroup of level $p^N$ of $\GSp_4(\mathbf{Z}_p)$.  The vector $v$ is stabilized by some $K_{N}$, $N \gg 0$.  We may choose $\eta$ in $C_c^\infty(\mathbf{Q}_p)$ so that $\hat{\eta}(y) = \charf( y \in 1 + p^N\mathbf{Z}_p)$.  In fact, $\eta(z) = \psi(-z)|p|^N \charf(p^N z \in \mathbf{Z}_p)$ is the unique such $\eta$.

The vector $v_1$, defined as in (\ref{v1Def}), is stable by some congruence subgroup $K_{N_1}$.  The stabilizer of the action $f_2 \mapsto f_2 g$ for $g$ in $\SL_{2,L}(\Q_p)$ is exactly $N(\Q_p)$, so we may select $\Phi_p$ such that the function $x \mapsto \Phi_p(f_2 x)$ is the restriction to $\SL_{2,L}(\Q_p)$ of the characteristic function of $N(\Q_p)K_{N_1}$.  Choosing such a $\Phi_p$ and putting it into (\ref{badfte2}), we obtain
\begin{align*} J_p(\ell,v_1, s) &= \int_{T'(\Q_p)}|\nu(t)|^{2s-2}\int_{N(\Q_p) \backslash \SL_{2,L}(\Q_p)}{\charf(x \in N(\Q_p)K_{N_1}) \ell(\pi(tx) v_1) \,dx}\,dt \\ &= C' \int_{T'(\Q_p)}{|\nu(t)|^{2s-2} \ell(\pi(t) v_1)\,dt} \\ &= C  \int_{\GL_1(\Q_p)}{|y|^{2s-2}\charf(y \in 1 + p^N \Z_p) \ell(\pi(t_y) v)\,dy} \\ &= C \ell(v) \end{align*}
for some positive constants $C'$ and $C$ that do not depend on $\ell$.  Here $t_y = \diag(y,y,1,1)$.  Choosing $\beta_p$ so that $\beta_p * v=  C^{-1} v_1$ gives the proposition. \end{proof}

\begin{proof}[Proof of Proposition \ref{KS:ram}]  Pick $\alpha_p$ so that $\alpha_p(v_Dg) = \charf(g \in {\GL_{2,L}^*} K_N)$ for $N$ sufficiently large.  Then we are reduced to the case in Proposition \ref{KS:ram1}. 
\end{proof}

The following lemma immediately implies Proposition \ref{KS:arch}.
\begin{lemma} Suppose that $v$ is a smooth vector in $V_{\pi_{\infty}}$, $\ell$ is a $(U,\chi)$-functional on the space of such smooth vectors, and $\ell(v) \neq 0$.  Then there exists a smooth vector $v'$ in $V_{\pi_{\infty}}$ and a Schwartz function $\Phi_\infty$ so that the integral $J_\infty(\ell,v',s)$ is not identically zero.\end{lemma}
\begin{proof} Choose a Schwartz function $\eta$ on $\R$ with $\hat{\eta}$ compactly supported near $y = 1$, and define $v_1$ as in (\ref{v1Def}).  Set $t_y = \diag(y,y,1,1)$.  Then
\begin{equation*} \int_{T'(\R)}{|\nu(t)|^{2s-2} \ell(\pi(t) v_1) \,dt} = \int_{\GL_1(\R)}{|y|^{2s-2}\hat{\eta}(y) \ell(\pi(t_y) v)\,dy} \end{equation*}
can be made nonzero for an appropriate $\eta$ since $\ell(v) \neq 0$.  With this choice of $v_1$, the inner integral in (\ref{badfte3}) is nonzero when $x=1$.  As in the proof of Proposition \ref{KS:ram1}, given any smooth function $\Phi'$ on $\SL_{2,L}(\R)$ that is left $N(\R)$-invariant and compactly supported on $N(\R) \backslash \SL_{2,L}(\R)$, we can pick $\Phi_\infty$ such that $\Phi_\infty(f_2 x) =\Phi'(x)$.  For $\mathrm{Re}(s) \gg 0$, the integrand in (\ref{badfte3}) is continuous as a function of $x \in \SL_{2,L}(\R)$.  Hence we can choose $\Phi_\infty$ so that (\ref{badfte3}) is nonzero.  Setting $v' = v_1$ proves the lemma. \end{proof}

%=======================================================
\section{Holomorphic Siegel modular forms}\label{hol}
%=======================================================
The purpose of this section is to prove Theorem \ref{archimedean}.  We choose the data $\alpha_\infty$ for $P_D^\alpha$ and $\Phi_\infty$ for $E(g,s,f^\Phi)$ so that the Archimedean integral $I_\infty(\phi,s)$ may be computed explicitly when $\phi$ comes from a level one holomorphic Siegel modular form.  We will see that for such a $\phi$ and with our choice of data, the Archimedean integral is equal to $\pi^{-2s}(4\pi)^{-(2s+r-2)}\Gamma(2s)\Gamma(2s+r-2)$ up to a constant. The discriminant $D$ is negative throughout this section, so $L = \mathbf{Q}(\sqrt{D})$ is an imaginary quadratic field.  We use the notation $u(X) = \mm{\mathbf{1}_2}{X}{}{\mathbf{1}_2}$ for a symmetric $2 \times 2$ matrix $X$ in this section.  The paper \cite{kudla} of Kudla was very helpful for us in making the correct choice of $\alpha_\infty$.

\subsection{Preliminaries}
Denote by $\mathcal{H}_2$ the upper half space of symmetric $2\times 2$ complex matrices with positive definite imaginary part.  We write $\GSp_4^+(\mathbf{R})$ for the elements of $\GSp_4(\mathbf{R})$ with positive similitude.  Then $\GSp_4^+(\mathbf{R})$ acts transitively on $\mathcal{H}_2$ via the formula
\[\mb{A}{B}{C}{D} z = (Az+B)(Cz+D)^{-1}.\]
We set $i = \sqrt{-1}$ and, abusing notation, also denote by $i$ the $2 \times 2$ matrix $i \mathbf{1}_2 \in \mathcal{H}_2$.  Denote by $K_\infty$ the subgroup of $\Sp_4(\mathbf{R})$ that stabilizes $i$.  Then
\[K_\infty = \left\lbrace \mb{A}{B}{-B}{A}: A + iB \in U(2)\right\rbrace.\]
For $\gamma \in \GSp_4(\mathbf{R})$, $\gamma = \mm{A}{B}{C}{D}$ and $z \in \mathcal{H}_2$, we write $j(\gamma,z) = \det(Cz+D)$.  In this section, we assume our cusp form $\phi$ satisfies
\begin{enumerate}
\item $\phi(gk_fk_\infty) = j(k_\infty,i)^{-r}\phi(g)$ for all $k_f \in \prod_{p < \infty}{\GSp_4(\mathbf{Z}_p)}$ and $k_\infty \in K_\infty$ and
\item the function $f_\phi: \mathcal{H}_2 \rightarrow \mathbf{C}$ defined by $f_\phi(g_\infty(i)) = \nu(g_\infty)^{-r}j(g_\infty,i)^r\phi(g_\infty)$ for all $g_\infty \in \GSp_4^+(\mathbf{R})$ is a classical holomorphic Siegel modular form.
\end{enumerate}
Assumption (1) implies that the function $f_\phi$ is well-defined.  Assumption (2) ensures that $f_\phi$ is a holomorphic function on $\mathcal{H}_2$.  It is then a consequence of assumption (1) that it is a classical Siegel modular form of level one and weight $r$.  For a general discussion about the relationship between Siegel modular forms and automorphic representations, see \cite[\S 4]{asgariSchmidt}.

Since $f_\phi$ is a level one Siegel modular form, it has a Fourier expansion of the form
\[f_\phi(Z) = \sum_{S > 0}{a(S) e^{2 \pi i \tr(S Z)}}.\]
Here the sum is taken over $2 \times 2$ symmetric matrices $S$ that are half-integral and positive definite, and the Fourier coefficients $a(S)$ are in $\C$.  Set
\begin{equation} \label{eqn:tdef} T = \begin{cases} \mm{-D}{}{}{1} & D \not \equiv 1 \pmod{4} \\ \mm{\frac{1-D}{4}}{\frac{1}{2}}{\frac{1}{2}}{1} & D \equiv 1 \pmod{4}\end{cases}.\end{equation}
The matrix $T$ is associated to $\chi$ in the sense that $\chi(u(X)) = \psi(\tr(TX))$.  It follows from assumptions (1) and (2) that for $g \in \GSp_4^+(\mathbf{R})$,
\begin{equation} \label{eqn:phiinfty}\phi_\chi(g_\infty) = a(T)\nu(g)^r j(g,i)^{-r}e^{2\pi i \tr(T g(i))}.\end{equation}
Indeed, suppose $u(X) = u(X_\infty) u(X_f)$, and $u(X_f) \in \GSp_4(\widehat{\Z})$.  Then applying the two assumptions,
\begin{equation}\label{FC1} \nu(g)^{-r}j(g,i)^r \phi(u(X) g) = \sum_{S > 0}{e^{2\pi i \tr(S X_\infty)} a(S) e^{2 \pi i \tr(S g \cdot i)}}.\end{equation}
The natural inclusion
\[U_P(\Z) \backslash U_P(\R)U_P(\widehat{\Z}) \rightarrow U_P(\Q) \backslash U_P(\A)\]
is a bijection.  Applying (\ref{FC1}), one obtains
\begin{align}\label{FC2} \nu(g)^{-r}j(g,i)^r \phi_\chi(g_\infty) &= \int_{U_P(\Z) \backslash U_P(\R)U_P(\widehat{\Z})}{\psi^{-1}(\tr(T X)) \left( \sum_{S > 0}{e^{2\pi i \tr(S X_\infty)} a(S) e^{2 \pi i \tr(S g \cdot i)}}\right)\,dX} \nonumber \\ &= a(T) e^{2 \pi i \tr(T g\cdot i)}. \end{align}

If $g_\infty \in \GSp_4(\R)$ has $\nu(g_\infty) < 0$, then our assumptions imply $\phi_\chi(g_\infty) = 0$.  To see this identity, first set $g'_\infty = \mm{-1}{}{}{1}_\infty g_\infty$, so that $\nu(g_\infty') > 0$.  Then
\begin{align*} \phi_\chi(g_\infty) &= \int_{U_P(\Q)\backslash U_P(\A)}{\chi^{-1}(u(X))\phi(u(X) g_\infty)\,dX} \\ &= \int_{U_P(\Q)\backslash U_P(\A)}{\chi^{-1}(\mm{-1}{}{}{1}u(X)\mm{-1}{}{}{1})\phi(\mm{-1}{}{}{1}u(X)\mm{-1}{}{}{1} g_\infty)\,dX} \\ & =  \int_{U_P(\Q)\backslash U_P(\A)}{\chi^{-1}(u(-X))\phi(u(X)g_\infty')\,dX} \\ &= \int_{U_P(\Q)\backslash U_P(\A)}{\psi^{-1}(\tr(-TX))\phi(u(X) g_\infty')\,dX}, \end{align*}
where we have used that $\phi$ is invariant on the right by $\GSp_4(\widehat{\Z})$.  The vanishing of the quantity on the last line follows as in the computation (\ref{FC2}) since $f_\phi(Z)$ has no negative definite Fourier coefficients.

We define the $K_\infty$-invariant inner product $\{\cdot,\cdot\}_4$ on $W_4(\mathbf{R})$ to be the unique inner product for which $e_1, e_2, f_1,f_2$ is an orthonormal basis.  Denote by $\{\,,\,\}_5$ the inner product on $\wedge^2 W_4(\R) \supseteq V_5(\R)$ induced by $\{\cdot,\cdot\}_4$.  That is, if $u_1, u_2, v_1, v_2$ are in $W_4(\R)$, then we set
\[\{u_1 \wedge u_2, v_1 \wedge v_2 \}_5 = \{u_1, v_1\}_4\{u_2, v_2\}_4 -\{u_1,v_2\}_4\{u_2,v_1\}_4\]
and extend $\{ \cdot ,\cdot \}_5$ by linearity.  Since $\{ \cdot,\cdot\}_5$ is induced by $\{ \cdot,\cdot\}_4$, $\{ \cdot,\cdot\}_5$ is also $K_\infty$-invariant.  Finally, denote by $\| \cdot \| $ the norms on $W_4(\R)$ and $V_5(\R)$ associated to these inner products.  That is, $\|v\|^2 = \{v,v\}_4$ or $\{v,v\}_5$.  One has
\[\|a_1e_1 + a_2e_2 + b_1f_1 + b_2f_2\|^2 = a_1^2 + a_2^2 + b_1^2 + b_2^2\]
and
\[ \| A e_1 \wedge e_2 + B_1e_1 \wedge f_2 + B_2(e_1 \wedge f_1 - e_2 \wedge f_2) + B_3 f_1 \wedge e_2 + C f_1 \wedge f_2\|^2 = A^2+B_1^2+2B_2^2+B_3^2 +C^2.\]

\subsection{Definition of and results on $P_D^\alpha$}\label{P_D hol}
We now define $\alpha_\infty$.  Let $r \ge 6$ be an integer.  Define a vector 
\begin{equation} w = -(e_1 - if_1) \wedge (e_2 -if_2) \in V_5(\mathbf{C}).\end{equation}
Then for $g \in \GSp_4(\mathbf{R})$ we define
\begin{equation} \label{alphainftydef} \alpha_\infty(g) =\overline{ (w,v_Dg)}^{-r}.\end{equation}
Recall that $v_D$ is defined by (\ref{vdnot1mod4}) or (\ref{vd1mod4}) depending on $D \pmod{4}$ and that $( \cdot, \cdot)$ is the $\GSp_4$-invariant linear form on $V_5$ defined in Section \ref{subsec:groupdefs}.  We will see shortly that the hypothesis $q(v) = \frac{1}{2}(v,v) < 0$ implies $(w,v) \neq 0$, and hence $\alpha_\infty$ is well-defined.

Recall that $\alpha = \prod_v{\alpha_v}$ is assumed to be factorizable.  Define $\alpha_f = \prod_{v < \infty}{\alpha_v}$ to be the finite part so that $\alpha = \alpha_f \otimes \alpha_\infty$.  With this definition of $\alpha_\infty$, one obtains
\[P_D^\alpha(g) = \sum_{\gamma \in \GL_{2,L}^*(\mathbf{Q}) \backslash \GSp_4(\mathbf{Q})}{\alpha_f(v_D \gamma g_f)\overline{(w,v_D\gamma g_\infty)}^{-r}}.\]
Note that $\alpha_\infty$ is not a Schwartz function on $V_5(\mathbf{R})$, but we will shortly see that the condition $r \geq 6$ implies that $P_D^\alpha$ is absolutely convergent and defines a function of moderate growth. The following lemma summarizes the properties of $w$ that we need.
\begin{lemma}\label{w properties} The element $w = -(e_1 - if_1) \wedge (e_2 -if_2) \in V_5(\mathbf{C})$ has the following properties.
\begin{enumerate}
\item $w k_\infty = j(k_\infty,i)^{-1}w$ for all $k_\infty \in K_\infty$.
\item For $v \in V_5(\mathbf{R})$, $|(w,v)|^2 = \|v\|^2 - (v,v)$, and hence $(w,v)$ is not zero when $q(v) < 0$.
\item Suppose that $g \in \GSp_4^+(\mathbf{R})$ and $v \in V_5(\mathbf{R})$.  Then $j(g,i)^{-1}\nu(g)(w,vg)$, as a function of $g$, is right-invariant under $Z(\mathbf{R})K_\infty$, and hence descends to a function of $\mathcal{H}_2$.  If
\begin{equation} \label{eqn:vdef} v = A e_1 \wedge e_2 + B_1 e_1 \wedge f_2 + B_2(e_1 \wedge f_1 - e_2 \wedge f_2) + B_3 f_1 \wedge e_2 + C f_1 \wedge f_2,\end{equation}
and $Z = g(i)$, then
\begin{equation}\label{eqn:gvpairing} j(g,i)^{-1}\nu(g)(w,vg) = -A\det(Z)+\tr\left(\mb{-B_1}{B_2}{B_2}{-B_3}Z\right) - C.\end{equation}
\end{enumerate}
\end{lemma}
\begin{proof} The first item is a simple calculation.  For the second, assume $v$ is written in the form (\ref{eqn:vdef}).  Then one computes
\[(w,v) = (C-A) + i(B_1-B_3) \textrm{ and } (v,v) = 2B_2^2 - 2B_1B_3 + 2AC,\]
from which the desired equality follows.

For the third item, we observe using the first part that $\nu(g)j(g,i)^{-1}wg^{-1}$ is a right $Z(\mathbf{R})K_\infty$-invariant function and thus descends to $\mathcal{H}_2$.  Hence to compute it, we may assume $g$ is in the Siegel parabolic, i.e.\ has the lower left $2\times 2$ block equal to zero.  We will compare the expressions $g(i) = Z$ and $\nu(g)j(g,i)^{-1}wg^{-1}$ to obtain the formula
\begin{align}
	\label{eqn:rstarginv} &\nu(g)j(g,i)^{-1}wg^{-1} =\\
	\nonumber &- \det(Z) f_1 \wedge f_2 + z_{11} f_1 \wedge e_2 + z_{12}(e_1\wedge f_1 - e_2 \wedge f_2) + z_{22} e_1 \wedge f_2 -e_1 \wedge e_2 .
\end{align}
The third item follows immediately from this equality.  To prove (\ref{eqn:rstarginv}), we first write a general element of the Siegel parabolic in the form $g=nm$, where these matrices are defined by
\[g = \underbrace{\mb{\mathbf{1}_2}{X}{}{\mathbf{1}_2}}_{n}\cdot \underbrace{\mb{\lambda \mathbf{1}_2}{}{}{\mathbf{1}_2}\cdot \left(\begin{array}{cccc} a & b & & \\ c & d & &\\ & & d & -c \\ & & -b & a\end{array}\right)}_{m}\]
for $X = \mm{x_{11}}{x_{12}}{x_{12}}{x_{22}}$.  Let $\delta = \det\mm{a}{b}{c}{d}$.  Then $g(i) = X+iY = Z$, where
\[Y = \frac{\lambda}{\delta}\mb{a^2+b^2}{ac+bd}{ac+bd}{c^2+d^2}.\]
We also have $\nu(g)=\nu(m)=\lambda\delta$ and $j(g,i)=\det\mm{d}{-c}{-b}{a}=\delta$, so $j(g,i)^{-1}\nu(g)=\lambda$.

We need to calculate $\nu(g)j(g,i)^{-1}wg^{-1}=\lambda wm^{-1}n^{-1}$ in terms of $a,b,c,d,X,\delta,$ and $\lambda$.  Recalling from Section \ref{P_D} that $V_5 = \wedge_0^2 W_4 \otimes \nu^{-1}$, we calculate the action of $m^{-1}$ on $w \in V_5$ as
\[wm^{-1} = -((e_1 -if_1)\wedge (e_2-if_2)m^{-1})\cdot\nu(m^{-1})^{-1},\]
where $m^{-1}$ acts on $e_i$ and $f_i$ by right multiplication.  In the following, we will substantially simplify the calculation by writing $W$ for the coefficient of $f_1\wedge f_2$ since it can be identified using the other coefficients at the end.  Expanding the right hand side, we have
\begin{align*}
	wm^{-1} =& -((\lambda\delta)^{-1}(de_1-be_2)-i\delta^{-1}(af_1+cf_2)) \wedge ((\lambda\delta)^{-1}(-ce_1+ae_2)-i\delta^{-1}(bf_1+df_2))\nu(m)\\
	=&-(\delta\lambda)^{-1}((de_1-be_2) \wedge (-ce_1+ae_2) -i\lambda(af_1+cf_2)\wedge(-ce_1+ae_2)\\
	&-i\lambda(de_1-be_2)\wedge(bf_1+df_2)+Wf_1\wedge f_2)\\
	=&-(\delta\lambda)^{-1}(\delta(e_1\wedge e_2)-i\lambda(c^2+d^2)e_1\wedge f_2-i\lambda(ac+bd)(e_1 \wedge f_1 -e_2 \wedge f_2)\\
	&+i\lambda(a^2+b^2)(e_2\wedge f_1)+Wf_1\wedge f_2)\\
	=&-(\delta\lambda)^{-1}(\delta(e_1\wedge e_2)-i\delta y_{22}e_1\wedge f_2-i\delta y_{12}(e_1 \wedge f_1 -e_2 \wedge f_2)+i\delta y_{11}(e_2\wedge f_1)\\
	&+Wf_1\wedge f_2),
\end{align*}
where in the third equality we cancel many terms with opposite signs.  We now calculate
\begin{align*}
	\lambda wm^{-1}n^{-1} =&-(e_1\wedge e_2-iy_{22}e_1\wedge f_2-iy_{12}(e_1 \wedge f_1 -e_2 \wedge f_2)+iy_{11}e_2\wedge f_1+Wf_1\wedge f_2)\cdot n^{-1}\\
	=&-(e_1-x_{11}f_1-x_{12}f_2)\wedge(e_2-x_{12}f_1-x_{22}f_2)+iy_{22}(e_1-x_{11}f_1-x_{12}f_2)\wedge f_2\\
	&+iy_{12}((e_1-x_{11}f_1-x_{12}f_2)\wedge f_1 -(e_2-x_{12}f_1-x_{22}f_2) \wedge f_2)\\
	&+iy_{11}f_1\wedge (e_2-x_{12}f_1-x_{22}f_2)+Wf_1\wedge f_2\\
	=&Wf_1 \wedge f_2 + z_{11} f_1 \wedge e_2 + z_{12}(e_1\wedge f_1 - e_2 \wedge f_2) + z_{22} e_1 \wedge f_2 -e_1 \wedge e_2,
\end{align*}
where the last step involves many cancellations.

To determine $W$, we observe that $(w,w)=0$ and thus $\nu(g)j(g,i)^{-1}wg^{-1}$ is also isotropic for the pairing $(\cdot,\cdot)$.  Using the coefficients of the other terms in the expression for $\nu(g)j(g,i)^{-1}wg^{-1}$, we obtain $-2W+2z_{12}^2-2z_{11}z_{22}=0$, or $W = -\det Z$.  The third item now follows.
\end{proof}

Using Lemma \ref{w properties}.(2), we can check that the sum defining $P_D^\alpha$ converges absolutely and defines a function of moderate growth.  First note that there is a Schwartz-Bruhat function $\alpha'_f$ on $V_5(\A_f)$ that satisfies
\begin{itemize}
	\item $|\alpha_f(v)| \leq \alpha_f'(v)$ and
	\item $\alpha_f'$ is a constant times the characteristic function of an open compact subset of $V_5(\A_f)$.
\end{itemize}
Suppose that $g = g_f g_\infty$.  Consider the set $\Lambda'(g_f)$ consisting of those elements $v$ of $V_5(\Q)$ that satisfy $\alpha_f'(v g_f) \neq 0$.  Then $\Lambda'(g_f)$ is a lattice in $V_5(\R)$.  Hence
\begin{equation}\label{PD bound} |P_D^\alpha(g)| \leq \sum_{v \in \Lambda'(g_f)}{\frac{1}{\left(||vg_\infty||^2 + |D|\right)^{\frac{r}{2}}}}. \end{equation}
The right-hand side of (\ref{PD bound}) converges for $r \geq 6$, and this expression also shows that $P_D^\alpha(g_f g_\infty)$ is a function of moderate growth.

Suppose $\alpha_f = \prod_{v < \infty}{\alpha_v}$ is the characteristic function of $V_5(\widehat{\Z})$, and $\overline{P_D(Z)}$ denotes the associated Siegel modular form.  (We drop the $\alpha$ from the notation to emphasize that we have made a particular special choice of $\alpha$.)  Then in classical notation we have
\[\overline{P_D(Z)} = \sum_{\substack{v \in V_5(\mathbf{Z})\\q(v) = -|D|}} \frac{1}{Q_v(Z)^r}\]
with $Q_v$ defined as follows.  When $v$ is written in the form given in (\ref{eqn:vdef}) above (in terms of $A,B_1,B_2,B_3$, and $C$), then we set
\[Q_v(Z) = -A\det(Z)+\tr\left(\mb{-B_1}{B_2}{B_2}{-B_3}Z\right) - C,\]
which is the right-hand side of (\ref{eqn:gvpairing}).  The function $\overline{P_D(Z)}$ is the analogue on $\GSp_4$ of the modular and Hilbert modular forms defined by Zagier in \cite{zagier}.

The following lemma is the key to the calculation of $I_\infty(\phi,s)$ below.  It is the Archimedean analogue of Proposition \ref{KS:alphachi}.
\begin{lemma} \label{lem:alphainfcalc}  Assume that $g \in \GSp_4^+(\mathbf{R}) \cap M(\mathbf{R})$ and set $Y = g(i)= \mathrm{Im}(g(i))$.  Recall that $\alpha_\infty$ is defined in (\ref{alphainftydef}) and the function $\alpha_{\chi,\infty}$ is defined in terms of $\alpha_\infty$ in (\ref{eqn:alphachivdef}).  Also recall the $2 \times 2$ symmetric matrix $T$ defined in (\ref{eqn:tdef}).  We have
\begin{equation} \overline{\alpha_{\chi,\infty}(g)} = \nu(g)^rj(g,i)^{-r}\frac{(-2\pi i)^r}{(r-1)!}e^{-2\pi \tr(TY)}.\label{eqn:alphainfty}\end{equation}
\end{lemma}
\begin{proof} By definition,
\[\overline{\alpha_{\chi_\infty}(g)} = \int_{N(\mathbf{R}) \backslash U(\mathbf{R})}{e^{-2\pi i \tr(TX)}(w,v_Du(X)g)^{-r}\,dX}.\]
We have
\begin{align*}
  &\nu(g)^{-r}j(g,i)^{r}\int_{N(\mathbf{R}) \backslash U(\mathbf{R})}{e^{-2\pi i \tr(TX)}(w,v_Du(X)g)^{-r}\,dX}\\
  = &\int_{N(\mathbf{R}) \backslash U(\mathbf{R})}{e^{-2\pi i \tr(TX)}(w,v_Du(X)g)^{-r}\nu(u(X)g)^{-r}j(u(X)g,i)^{r}\,dX}
\end{align*}
since $\nu(u(X)) = 1$ and $j(u(X)g,i) = j(g,i)$.  By the third part of Lemma \ref{w properties}, 
\[(w,v_Du(X)g)^{-r}\nu(u(X)g)^{-r}j(u(X)g,i)^{r} = \tr(T(u(X)g)(i))^{-r}.\]
We define $z = \tr(T(u(X)g)(i))$ and define $x$ and $y$ by $z = x +iy$, noting that $y > 0$.  Then we have
\begin{align*} & \frac{j(g,i)^{r}}{\nu(g)^r}\int_{N(\mathbf{R}) \backslash U(\mathbf{R})}{e^{-2\pi i \tr(TX)}(w,v_Du(X)g)^{-r}\,dX} = \int_{\mathbf{R}}{e^{-2\pi i x}\frac{1}{(x+iy)^r}\,dx} \\
= & e^{-2\pi \tr(TY)}\int_{\mathrm{Im}(z) = y}{e^{-2\pi i z}z^{-r}\,dz} 
= e^{-2\pi \tr(TY)}\frac{(-2\pi i)^r}{(r-1)!}.\end{align*}
The last equality is standard and comes from computing the residue of the integrand around $0$.  The lemma follows.
\end{proof}

\subsection{The Eisenstein series and the calculation of $I_\infty$}
For $w \in W_4(\mathbf{R})$, we define $\Phi_\infty(w) = e^{-\pi ||w||^2}$.  Then for $g \in \GSp_4(\mathbf{R})$,
\[f^\Phi_\infty(g,s) = |\nu(g)|^{2s}\int_{\GL_1(\mathbf{R})}{\Phi_\infty(tf_2g)|t|^{4s}\,dt}.\]
Then $f_\infty^\Phi$ is right invariant under $Z(\mathbf{R})K_\infty$, hence descends to a function on $\mathcal{H}_2$ for $g$ in $\GSp_4^+(\mathbf{R})$.  In fact, if $g \in \GSp_4^+(\mathbf{R})$ and $g(i) = X + iY$, then
\begin{equation} \label{eqn:finfty} f^\Phi_\infty(g,s) = \pi^{-2s}\Gamma(2s)|y_{11}|^{-2s}|\det(Y)|^{2s},\end{equation}
where $y_{ij}$ denotes the entry of $Y$ in position $(i,j)$.

We now have all the ingredients necessary to compute the Archimedean integral
\[ I_\infty(\phi,s) = \int_{Z(\mathbf{R})N(\mathbf{R}) \backslash \GSp_4(\mathbf{R})}{f_\infty^\Phi(g,s)\alpha_\infty(g)\phi_\chi(g)\,dg}.\]
We will compute it up to a nonzero constant.  Integrating over $N(\mathbf{R}) \backslash U(\mathbf{R})$ this becomes 
\[ \int_{Z(\mathbf{R})\backslash M(\mathbf{R})}{\delta_P^{-1}(m) f_\infty^\Phi(m,s)\alpha_{\chi,\infty}(m)\phi_\chi(m)\,dm}.\]
As all the terms in the integrand are right invariant under $K_\infty$, and since $\phi_\chi(m) = 0$ if $\nu(m) < 0$, we may rewrite this as an integral over the purely imaginary part of the upper half space.  We first use the formulas (\ref{eqn:phiinfty}), (\ref{eqn:finfty}), and (\ref{eqn:alphainfty}) to obtain
\[ a(T)\frac{(2\pi i)^r}{(r-1)!}\pi^{-2s}\Gamma(2s)\int_{Z(\mathbf{R})\backslash M^+(\mathbf{R})}{\delta_P^{-1}(m) |\det(Y)/y_{11}|^{2s} \left|\frac{\nu(m)}{j(m,i)}\right|^{2r}e^{-4\pi \tr(TY)}\,dm}.\]
Set $dY = dy_{11}dy_{12}dy_{22}$.  Then we can replace $\delta_P(m)dm$ with $dY$ and use the relations $\delta_P(m) = |\det(Y)|^{3/2}$ and $|\nu(m)|^{2r}|j(m,i)|^{-2r} = |\det(Y)|^r$ to transform $I_\infty(\phi,s)$ into, up to a nonzero constant,
\[a(T)\pi^{-2s}\Gamma(2s)\int_{Y}{|\det(Y)/y_{11}|^{2s+r-3}y_{11}^{r-3}e^{-4\pi \tr(TY)}dY}\]
where is the integration is over the set of two-by-two real symmetric positive-definite matrices.  Following Kohnen-Skoruppa \cite[pg. 548]{ks}, we make the variable change $t = \det(Y)/y_{11}$ $= y_{22}-y_{12}^2/y_{11}$.  We obtain
\[a(T)\pi^{-2s}\Gamma(2s)\int_{t > 0, y_{11} > 0, y_{12} \in \mathbf{R}}{t^{2s+r-2}e^{-4\pi(t+y_{12}^2/y_{11}+|D|y_{11})}\,dy_{11}dy_{12}\frac{dt}{t}}\]
or
\[a(T)\pi^{-2s}\Gamma(2s)\int_{t > 0, y_{11} > 0, y_{12} \in \mathbf{R}}{t^{2s+r-2}e^{-4\pi(t+y_{12}^2/y_{11}+y_{12}+|D|y_{11})}\,dy_{11}dy_{12}\frac{dt}{t}}\]
if $D \not \equiv 1 \pmod{4}$ or $D \equiv 1 \pmod{4}$, respectively.  Up to a nonzero constant (depending on $|D|$), this is
\[a(T)\pi^{-2s}(4\pi)^{-(2s+r-2)}\Gamma(2s)\Gamma(2s+r-2).\]
This completes the proof of Theorem \ref{archimedean}.
%=====================================================
%Appendix
%=====================================================
\appendix
\section{The local Spin $L$-function}\label{localSpin}
In this appendix we recall the definition of the local Spin $L$-function and prove Proposition \ref{omega_p}.  Since we work primarily with a Borel subgroup of $\GSp_4$, we reorder the symplectic basis so that we may choose the Borel to be upper-triangular.  Namely, we use the ordered basis $(e_1, e_2, f_2, f_1)$, where as usual the pairing is defined by $\langle e_i,f_j \rangle = \delta_{ij}$.  The subgroup $B$ of upper-triangular matrices in $\GSp_4$ is a Borel.  Denote by $T$ and $N$ the diagonal torus and unipotent radical of $B$, respectively.

Suppose that $\alpha$ is an unramified character of $T$.  Also denote by $\alpha$ the extension of this character to $B$.  Let $\pi$ be the unramified irreducible subquotient of $\ind_{B(\Q_p)}^{\GSp_4(\Q_p)}(\delta_B^{1/2} \alpha)$, where $\delta_B$ denotes the modulus character of $B$.  We define four elements of $T(\mathbf{Q}_p)$ of similitude $p$ by
\[t_A = \diag(p,p,1,1), \; t_B = \diag(p,1,p,1), \; t_C = \diag(1,p,1,p),\textrm{ and} \; t_D = \diag(1,1,p,p).\]
The local Spin $L$-function of $\pi$ is defined by
\begin{equation}\label{SpinL}L(\pi,\Spin,s)=\left(1 - \alpha(t_A)|p|^s\right)^{-1}\left(1 - \alpha(t_B)|p|^s\right)^{-1}\left(1 - \alpha(t_C)|p|^s\right)^{-1}\left(1 - \alpha(t_D)|p|^s\right)^{-1}.\end{equation}
Since the Weyl group of $\GSp_4$ permutes $\set{t_A,t_B,t_C,t_D}$, the right-hand side of (\ref{SpinL}) is Weyl-invariant.  It follows that $L(\pi,\Spin,s)$ depends only on $\pi$ rather than $\alpha$. %reference: Thm 6.3.6 of Casselman's notes

Denote by $\phi_\alpha$ the unique spherical vector in $\ind_{B(\Q_p)}^{\GSp_4(\Q_p)}(\delta_B^{\frac{1}{2}} \alpha)$ that satisfies the normalization $\phi_\alpha(1) =1$.  The Macdonald spherical function $\Omega_\alpha$ is defined by
\begin{equation}\label{omegaDef} \Omega_\alpha(g) = \int_{\GSp_4(\Z_p)}{\phi_\alpha(kg) \, dk}. \end{equation}
%We will prove Proposition \ref{omega_p} by computing explicitly with the left-hand side of (\ref{phiAlpha}). 

\begin{proof}[Proof of Proposition \ref{omega_p}] Recall the function $\Delta^s(g) = |\nu(g)|^s \charf(g)$.  Since $\Delta^s$ is invariant on the left by $\GSp_4(\Z_p)$,
\begin{equation}\label{phiAlpha} \int_{\GSp_4(\Q_p)}{\phi_{\alpha}(g) \Delta^s(g)\, dg} = \int_{\GSp_4(\Q_p)}{\Omega_\alpha(g) \Delta^s(g)\,dg}. \end{equation}
Applying the Iwasawa decomposition to the left-hand side of (\ref{phiAlpha}), one obtains
\begin{equation}\label{Ib0} \int_{\GSp_4(\Q_p)}{\phi_\alpha(g) \Delta^s(g) \, dg} = \int_{T(\Q_p)}{\delta_B^{-\frac{1}{2}}(t) \alpha(t) I_B(t) |\nu(t)|^s \, dt} \end{equation}
where
\begin{equation*} I_B(t) = \int_{N(\Q_p)}{\charf(nt)\,dn}. \end{equation*}
We will compute $I_B(t)$ explicitly.

We first compute a simple function related to $I_B$.  Let $P$, $U$, and $M$ denote the Siegel parabolic, its unipotent radical, and its Levi, respectively. For $m = \mm{m_1}{}{}{m_2} \in M(\mathbf{Q}_p)$, we define
\begin{equation*} I_P(m) = \int_{U(\Q_p)}{\charf(um)\,du}. \end{equation*}
We write $u(X)$ for the matrix $\mm{\mathbf{1}_2}{X}{}{\mathbf{1}_2}$ for $X \in M_2(\mathbf{Q}_p)$.  If $X = \mm{x_{11}}{x_{12}}{x_{21}}{x_{22}}$, the condition $u(X) \in U(\mathbf{Q}_p)$ is equivalent to $x_{11}=x_{22}$.  We define a function $\Delta_1:M_2(\mathbf{Q}_p) \rightarrow \mathbf{R}_{\ge 0}$ by $\Delta_1(\mm{a}{b}{c}{d}) = \sup\{|a|, |b|, |c|, |d|\}$.
\begin{lemma} \label{lemma:ipformula} Suppose $m = \mm{m_1}{}{}{m_2} \in M_4(\Z_p)$.  Then $I_P(m) = |\det(m_2)|^{-1} \Delta_1(m_2)^{-1}$. \end{lemma}
\begin{proof} Since $m \in M_4(\Z_p)$, $I_P(m)$ is the measure of the set $H(m_2)$ defined by
\[H(m_2) = \{X \in M_2(\Q_p): u(X) \in U(\mathbf{Q}_p) \text{ and } X m_2 \in M_2(\Z_p)\}. \]
It is clear that $H(m_2 k') = H(m_2)$ if $k' \in \GL_2(\Z_p)$.  If $k \in \GL_2(\Z_p)$ and ${}^ak$ denotes interchanging the top left and bottom right entries, then the map $X \mapsto {}^ak X k$ defines a measure-preserving bijection $H(km_2) \rightarrow H(m_2)$.   Suppose $m_2 = k \mm{p^a}{}{}{p^b} k'$ with $k, k'$ in $\GL_2(\Z_p)$ and $a \geq b \geq 0$ integers.  It follows that $I_P(m_2) = I_P\left(\mm{p^a}{}{}{p^b}\right)$.  The set
\[H\left(\mm{p^a}{}{}{p^b}\right) = \mb{p^{-b}\Z_p}{p^{-b}\Z_p}{p^{-a}\Z_p}{\ast}\]
has measure $p^{a+2b}$, so we have $I_P(m) = p^{a+2b} = |\det(m)|^{-1} p^b$.  Since $\Delta_1(m)$ is a bi-$\GL_2(\Z_p)$-invariant function on $M_2(\Q_p)$, $\Delta_1(m_2) = p^{-b}$.  The lemma follows. \end{proof}

The integrals $I_B$ and $I_P$ are related by
\begin{equation}\label{I_B I_P} I_B(t) = \int_{U(\Q_p) \backslash N(\Q_p)}{I_P(nt)\, dn}. \end{equation}
Suppose $t = \diag(\lambda t_1, \lambda t_2, t_1, t_2) \in M_4(\Z_p) \cap T(\Q_p)$.  Identifying $U(\Q_p) \backslash N(\Q_p)$ with $\mathbf{G}_a(\Q_p)$ in the right-hand-side of (\ref{I_B I_P}) and applying Lemma \ref{lemma:ipformula}, we obtain
\begin{align}\label{Ib2} I_B(t) &= \int_{\mathbf{G}_a(\Q_p)}{|t_1 t_2|^{-1} \sup\{|t_1|, |t_2|, |xt_2|\}^{-1} \charf( |x\lambda t_2| \leq 1) \charf(|x t_2| \leq 1)\, dx} \nonumber \\ & = |t_1 t_2^2|^{-1} \int_{\mathbf{G}_a(\Q_p)}{\sup\{|t_1|, |t_2|, |y|\}^{-1} \charf( |\lambda y| \leq 1) \charf(|y| \leq 1)\, dy} \nonumber \\ &= \delta_B(t)^{1/2}|\nu(t)|^{-3/2} \int_{\mathbf{G}_a(\Q_p)}{\sup\{|t_1|, |t_2|, |y|\}^{-1} \charf( |\lambda y| \leq 1) \charf(|y| \leq 1)\, dy}. \end{align}
For the second equality, we use the change of variables $xt_2 = y$.  We have used $\delta_B(t) = |\lambda|^3|t_1/t_2|$ and $\nu(t) = \lambda t_1 t_2$ for the third equality.

If $\mu \in \Q_p^\times$ and $|\mu| \leq 1$, then
\begin{align}\label{Ib3} \int_{\Q_p}{\sup\{|\mu|,|y|\}^{-1} \charf(|y| \leq 1) \,dy}  &= 1 + \ord_p(\mu)(1-|p|) = \ord_p(p \mu) - |p|\ord_p(\mu). \end{align}
For $t \in T$, denote by $\val_p(t)$ the largest integer such that $p^{-\val_p(t)} t \in M_4(\Z_p)$.  To evaluate (\ref{Ib2}), we consider the cases $|\lambda| > 1$ and $|\lambda| \leq 1$.  If $|\lambda| \leq 1$, then by applying (\ref{Ib3}) to (\ref{Ib2}) one obtains
\[I_B(t) = \delta_B(t)^{1/2}|\nu(t)|^{-3/2} \left(\val_p(pt) - |p|\val_p(t)\right).\]
If $|\lambda| > 1$, then
\begin{align}\label{Ib4} I_B(t) &= \delta_B(t)^{1/2}|\nu(t)|^{-3/2} \int_{\mathbf{G}_a(\Q_p)}{\sup\{|t_1|, |t_2|, |y|\}^{-1} \charf( |\lambda y| \leq 1)\, dy} \\ \label{Ib5} &= \delta_B(t)^{1/2}|\nu(t)|^{-3/2} \int_{\mathbf{G}_a(\Q_p)}{\sup\{|\lambda t_1|, |\lambda t_2|, |z|\}^{-1} \charf( |z| \leq 1)\, dz} \\ &= \delta_B(t)^{1/2}|\nu(t)|^{-3/2} \left(\val_p(pt) - |p|\val_p(t)\right) \nonumber \end{align}
using the variable change $z = \lambda y$ to go from (\ref{Ib4}) to (\ref{Ib5}).  Hence
\begin{equation} \label{Ib6} I_B(t) = \delta_B(t)^{1/2}|\nu(t)|^{-3/2} \left(\val_p(pt) - |p|\val_p(t)\right) \end{equation}
in both cases.  

Inserting (\ref{Ib6}) into (\ref{Ib0}), we obtain
\begin{equation}\label{eqn:gsp4int} \int_{\GSp_4(\Q_p)}{\phi_\alpha(g) \Delta^s(g) \, dg} = \int_{T(\Q_p)}{ \alpha(t) \charf(t)\left(\val_p(pt) - |p|\val_p(t)\right) |\nu(t)|^{s-3/2} \, dt}. \end{equation}
Let $\omega_\alpha$ denote the central character of $\alpha$.  Then
\begin{align}\label{eqn:torusint} \int_{T(\Q_p)}{ \alpha(t) \charf(t)\val_p(t)|\nu(t)|^{s-3/2} \, dt} &= \int_{T(\Q_p)}{ \alpha(pt) \charf(pt)\val_p(pt) |\nu(pt)|^{s-3/2} \, dt} \nonumber \\ &= \omega_\alpha(p)|p|^{2s-3} \int_{T(\Q_p)}{ \alpha(t) \charf(pt)\val_p(pt) |\nu(t)|^{s-3/2} \, dt} \nonumber \\ &=  \omega_\alpha(p)|p|^{2s-3} \int_{T(\Q_p)}{ \alpha(t) \charf(t)\val_p(pt)|\nu(t)|^{s-3/2} \, dt}. \end{align}
In the last line we have used $\charf(pt)\val_p(pt) = \charf(t)\val_p(pt)$, which holds because $\val_p(pt) = 0$ when $pt \in M_4(\Z_p)$ and $t \notin M_4(\Z_p)$.  Inserting (\ref{eqn:torusint}) into (\ref{eqn:gsp4int}), we obtain
\begin{equation*}\int_{\GSp_4(\Q_p)}{\phi_\alpha(g) \Delta^s(g) \, dg} = \left(1-\omega_\alpha(p)|p|^{2s-2}\right)\int_{T(\Q_p)}{ \alpha(t) \charf(t)\val_p(pt) |\nu(t)|^{s-3/2} \, dt}. \end{equation*}

\begin{lemma}\label{lemma2} Suppose $t \in T(\Q_p) \cap M_4(\Z_p)$.  If $\val_p(t) = k$, then up to the action of $T(\Z_p)$, $t$ can be written as a product of $t_A, t_B, t_C$ and $t_D$ in $k+1$ ways. \end{lemma}
\begin{proof}The set of elements of $T(\Q_p) \cap M_4(\Z_p)/T(\Z_p)$ with $\val_p(t) = k$ is given by
\begin{equation*}\label{t valp} \{\diag(p^{u_1},p^{u_2},p^{u_3}, p^{u_4}): u_1 + u_4 = u_2 + u_3, \min\{u_1, u_2, u_3, u_4\} = k \}. \end{equation*}
For such an element, we are asked to find the number of $4$-tuples of nonnegative integers $(\alpha, \beta, \gamma, \delta)$ so that
\begin{equation}\label{tAtB} t_A^\alpha t_B^\beta t_C^\gamma t_D^\delta = \diag(p^{u_1},p^{u_2},p^{u_3}, p^{u_4}).\end{equation}
Set $u_i' = u_i -k$.  There is a unique solution $(\alpha', \beta', \gamma', \delta')$ to the equation
\[t_A^{\alpha'} t_B^{\beta'} t_C^{\gamma'} t_D^{\delta'} = \diag(p^{u_1'},p^{u_2'},p^{u_3'}, p^{u_4'})\]
since some $u_i'=0$, which determines two of $\alpha', \beta', \gamma',$ and $\delta'$, and the other two are determined by two of the remaining values of $u_i'$.  We now prove by induction on $k$ that the solutions are
\begin{equation}\label{solns} (t_At_D)^x (t_B t_C)^y t_A^{\alpha'} t_B^{\beta'} t_C^{\gamma'} t_D^{\delta'} = \diag(p^{u_1},p^{u_2},p^{u_3}, p^{u_4}),\end{equation}
where $x, y \geq 0$ and $x + y = k$.  Assume the case $\val_p(t)=k$ is known, and consider the case $\val_p = k+1$.  Then for each tuple $(\alpha,\beta,\gamma,\delta)$, we must have either both $\alpha \ge 1$ and $\delta \ge 1$ or both $\beta \ge 1$ and $\gamma \ge 1$.  Dividing by $t_At_D$ or $t_Bt_C$ in these respective cases, we find by our inductive hypothesis that every solution must be the product of either $t_At_D$ or $t_Bt_C$ with the left-hand side of (\ref{solns}).  The $k+1$ unique tuples obtainable in this way are the ones enumerated in (\ref{solns}).

\end{proof}

Applying Lemma \ref{lemma2}, one obtains
\begin{align*}\int_{T(\Q_p)}{ \alpha(t) \charf(t)\val_p(pt) |\nu(t)|^{s-3/2} \, dt} &= \left(\sum_{k \geq 0} \alpha(t_A)|p|^{k(s-3/2)}\right) \left(\sum_{k \geq 0} \alpha(t_B)|p|^{k(s-3/2)}\right) \\ & \times \left(\sum_{k \geq 0} \alpha(t_C)|p|^{k(s-3/2)}\right) \left(\sum_{k \geq 0} \alpha(t_D)|p|^{k(s-3/2)}\right).\end{align*}
The proposition follows.
\end{proof}
\bibliography{integralRepnBibKS}
\bibliographystyle{abbrv}
\end{document}